\documentclass[10pt,a4paper]{amsart}

\usepackage{amssymb}
\usepackage[utf8]{inputenc}
\usepackage[english]{babel}
\usepackage{amsfonts}
\usepackage{amsmath}
\DeclareFontFamily{U}{mathx}{\hyphenchar\font45}
\DeclareFontShape{U}{mathx}{m}{n}{
      <5> <6> <7> <8> <9> <10>
      <10.95> <12> <14.4> <17.28> <20.74> <24.88>
      mathx10
      }{}
\DeclareSymbolFont{mathx}{U}{mathx}{m}{n}
\DeclareFontSubstitution{U}{mathx}{m}{n}
\DeclareMathAccent{\widecheck}{0}{mathx}{"71}
\DeclareMathAccent{\wideparen}{0}{mathx}{"75}

\usepackage{amsthm}
\usepackage{color}
\usepackage{enumitem}
\usepackage[nameinlink]{cleveref}

\newtheorem{main-theorem}{Theorem}
\newtheorem{proposition}{Proposition}[section]
\newtheorem{theorem}[proposition]{Theorem}

\newtheorem{corollary}[proposition]{Corollary}
\newtheorem{lemma}[proposition]{Lemma}
\theoremstyle{remark}
\newtheorem{remark}[proposition]{Remark}

\theoremstyle{definition}
\newtheorem{definition}[proposition]{Definition}
\newtheorem*{acknowledgements}{Acknowledgements}

\DeclareMathOperator{\supp}{supp}

\DeclareMathOperator{\graph}{graph}

\newcommand{\R}{\mathbb{R}}
\newcommand{\C}{\mathbb{C}}

\newcommand{\N}{\mathbb{N}}
\newcommand{\Sph}{\mathbb{S}}
\newcommand{\dd}{\mathrm{d}}
\newcommand{\tm}{\mathrm{t}}
\newcommand{\x}{\mathrm{x}}
\newcommand{\Id}{\mathrm{Id}}
\newcommand{\Orth}{\mathrm{O}}
\newcommand{\T}{\mathsf{T}}

\title[Inverse problems for data-driven prediction]{An inverse problem for data-driven prediction\\ in quantum mechanics}

\author{Pedro Caro}
\address{Pedro Caro\\
Ikerbasque \&
Basque Center for Applied Mathematics, Bilbao, Spain\\}
\email{pcaro@bcamath.org}

\author{Alberto Ruiz}
\address{Alberto Ruiz\\
Universidad Aut\'onoma de Madrid, 28049 Cantoblanco, Spain}
\email{alberto.ruiz@uam.es}

\begin{document}

\begin{abstract}
Data-driven prediction in quantum 
mechanics consists in providing an approximative description of
the motion of any particles at any 
given time, from data that have been previously collected for a certain number of 
particles under the influence of the same Hamiltonian. The difficulty of this 
problem comes from the ignorance of the exact Hamiltonian ruling the dynamic.
In order to address this problem, we formulate 
an inverse problem consisting in determining   
the Hamiltonian of a quantum system from the knowledge of the state at some 
fixed finite time for each initial state.
We focus on the simplest case where the 
Hamiltonian is given by $-\Delta + V$, where the potential
$V = V(\tm, \x)$ is non-compactly supported. Our main result is a uniqueness
theorem, which establishes that the Hamiltonian ruling the dynamic of all 
quantum particles is determined by the prescription of the initial and final
states of each particle. As a consequence, one expects to be able to know
the state of any particle at any given time, without 
an a priori knowledge of the
Hamiltonian just from the 
data consisting of the initial and final state of each
particle.
\end{abstract}

\date{\today}


\maketitle

\section{Introduction}
In quantum mechanics, the state of some moving particles
during an interval of time
$[0, T]$ is described by the family of wavefunctions 
$\{ u(t, \centerdot) : t \in [0, T] \} $, here
$u(t, \centerdot)$ denotes the function that takes the value $u(t, x) \in \C$ at
the point $x \in \R^n$. If the motion takes place 
under the influence of an electric potential $V = V(\tm, \x)$ and the initial 
state $u(0,  \centerdot)$ is prescribed by $f$, then the family of 
wavesfunctions satisfies
the initial-value problem for the Schr\"odinger equation
\begin{equation}
	\left\{
		\begin{aligned}
		& i\partial_\tm u = - \Delta u + V u & & \textnormal{in} \enspace (0,T) \times \R^n, \\
		& u(0, \centerdot) = f &  & \textnormal{in} \enspace \R^n.
		\end{aligned}
	\right.
	\label{pb:IVP}
\end{equation}
Thus, given an initial state $f$ we predict the motion of these 
quantum particles in the presence of $V$ by solving \eqref{pb:IVP}.
This classic approach
requires the exact knowledge of the Hamiltonian $-\Delta + V$ governing
the dynamic to predict the evolution of the system.
For this reason, it is usually referred to as 
\emph{phenonmenology-driven prediction}.

In the past two decades, we have witnessed how the use of data has become a 
game-changer for innovation and the development of new technology. Thus, it
is natural to investigate how data can be used to predict the evolution of
physical systems. For this reason, data-driven
approximation of solutions of partial differential equations is emerging as a 
powerful paradigm \cite{zbMATH07513856,zbMATH07516442}.

In this paper, we formulate a
data-driven prediction problem in quantum mechanics
and propose an inverse problem to provide a theoretical framework that explains
why one can solve the initial value problem \eqref{pb:IVP} 
ignoring the exact Hamiltonian if enough data, 
regarding initial and final states, is available.

In order to formulate the data-driven prediction problem in quantum mechanics,
we introduce the family of operators
$\{ \mathcal{U}_t : t \in [0, T] \}$ 
that, for every initial state $f$, is given by
\[ u (t, \centerdot) = \mathcal{U}_t f \]
for all $t \in [0, T]$, where $u$ solves \eqref{pb:IVP};
and for $N \in \N$ we define the 
\emph{initial-to-final-state set with $N$ elements}
\[ \mathcal{D} = \{ (f_1, \mathcal{U}_T f_1), \dots, (f_N, \mathcal{U}_T f_N) \}.\]
Then, the \emph{data-driven approximate prediction problem}
consists in analysing when it is possible
to compute, from the initial-to-final-state set $\mathcal{D}$,
a suitable approximation ---as $N$ increases--- of the wavefunctions 
$u(t, \centerdot) = \mathcal{U}_t f$, for all $t \in [0, T]$ and all initial
state $f$.
Note that this formulation assumes
that the Hamiltonian has a particular structure but ignores a full description of 
the electric potential $V$. The formulation of this problem is motivated by
applications, where data will be given by the set $\mathcal{D}$ with a very large
$N$. There are many non-trivial challenges around this question, 
however, in this article we focus on an exact version of this problem which will
establish a robust theoretical framework to address the approximate problem.

Recall that whenever $V \in L^1 ((0, T); L^\infty(\R^n))$ we have
that, for every $f \in L^2(\R^n)$, there exists a unique 
$u \in C([0,T]; L^2(\R^n))$ solving \eqref{pb:IVP} ---see for example 
\cite{zbMATH05013664,zbMATH02204588}.
Moreover, \textit{the evolution map}
\[ \mathcal{U} : f \in L^2(\R^n) \mapsto u \in C([0, T]; L^2(\R^n)) \]
is linear and bounded, and consequently we have that
\[ \mathcal{U}_t : f \in L^2(\R^n) \mapsto u(t, \centerdot) \in L^2(\R^n) \]
is also bounded for all $t \in [0,T]$.
The solutions of the form $u = \mathcal{U} f$ with $f \in L^2(\R^n)$ are called
\emph{physical}, while for the final time $T$ 
the operator $\mathcal{U}_T$ is referred to as 
the \textit{initial-to-final-state map}.
Note that the initial-to-final-state set $\mathcal{D}$ consists of $N$ points
in the graph of $\mathcal{U}_T$:
\[ \graph(\mathcal{U}_T) = \{ (f, \mathcal{U}_T f) : f \in L^2(\R^n) \}.\]
Then, the \textit{data-driven prediction problem} 
in quantum mechanics consists in computing the evolution map
$ \mathcal{U} $ from the initial-to-final-state map 
$\mathcal{U}_T$.

A fundamental step to solve the data-driven prediction problem is to
show that the initial-to-final-state map uniquely determines the evolution map.
The main result of this article is to show that this is the case whenever 
$n \geq 2$ and the electric potential 
satisfies an a priori assumptions about its 
decay at infinity.

\begin{definition} \label{def:exp_decay} Let $\Sigma \subset \R \times \R^n$ denote the open subset 
$(0, T) \times \R^n$. We say that 
$V \in L^1((0, T); L^\infty (\R^n))$ has \emph{super-exponential decay} if
$ e^{\rho |\x|} V \in L^2(\Sigma) \cap L^\infty(\Sigma) $ for all $\rho > 0$.
\end{definition}

The super-exponential decay is sufficient to ensure that the initial-to-final-state map uniquely determines the evolution map.

\begin{main-theorem}\label{th:uniqueness_data-driven} \sl
Consider $V_1$ and $V_2$ in $L^1((0, T); L^\infty (\R^n))$ with $n \geq 2$.
Let $\mathcal{U}_T^j$ and $\mathcal{U}^j$ with $j \in \{ 1, 2 \}$ denote 
the initial-to-final-state and evolution maps associated to the Hamiltonian
$-\Delta + V_j$. If 
$V_1$ and $V_2$ have super-exponential decay, then
\[ \mathcal{U}_T^1 = \mathcal{U}_T^2 \, \Rightarrow  \, \mathcal{U}^1 = \mathcal{U}^2. \] 

\end{main-theorem}

Our approach to solve the data-driven problem consists in analysing an
inverse problem whose goal is to determine uniquely 
the potential $V$ from its corresponding initial-to-final-state map 
$\mathcal{U}_T$, and reconstructing $V$ from $\mathcal{U}_T$ whenever 
this unique determination holds. It is straightforward to see that solving the
\textit{initial-to-final-state inverse problem} is sufficient to solve
the data-driven problem.

In this paper we prove a uniqueness theorem, 
for dimension $n \geq 2$, that consists in
determining the electric potential $V$ with super-exponential decay, from the 
knowledge of the final state for each initial 
state.

\begin{main-theorem}\label{th:uniqueness} \sl
Consider $V_1$ and $V_2$ in $L^1((0, T); L^\infty (\R^n))$ with $n \geq 2$.
Let $\mathcal{U}_T^j$ with $j \in \{ 1, 2 \}$ denote 
the initial-to-final-state map associated to the Hamiltonian
$-\Delta + V_j$. If 
$V_1$ and $V_2$ have super-exponential decay, then
\[ \mathcal{U}_T^1 = \mathcal{U}_T^2 \, \Rightarrow  \, V_1 = V_2. \] 
\end{main-theorem}

The \cref{th:uniqueness_data-driven} is a straight 
forward consequence of the \cref{th:uniqueness}.
Furthermore, note that a reconstruction algorithm based on the proof of our 
\cref{th:uniqueness} would yield a way of computing the evolution map 
$\mathcal{U}$ from the initial-to-final-state map $\mathcal{U}_T$, which would 
solve the data-driven prediction problem in its exact form.
Furthermore, such a reconstruction algorithm would
open the possibility to construct quantum Hamiltonians that would make many 
different initial states end at once in some specific final states
---whenever these initial and final states belong to $\graph(\mathcal{U}_T)$ 
for some $V$.

The scheme of the proof of the \cref{th:uniqueness} is as follows.
From the equality $\mathcal{U}_T^1 = \mathcal{U}_T^2$ 
we derive an orthogonality relation, that reads
\[ \int_\Sigma (V_1 - V_2)u_1 \overline{v_2} \, = 0 \]
for all physical solutions $u_1$ and $v_2$ in $C([0,T]; L^2(\R^n))$ 
of the equations
\[ (i\partial_\tm + \Delta - V_1) u_1 = 0  \enspace \textnormal{and} \enspace (i\partial_\tm + \Delta - \overline{V_2}) v_2 = 0 \enspace \textnormal{in} \enspace \Sigma. \]
With this identity at hand, and inspired by the method of complex geometrical 
optics (CGO for short) solutions for the Calderón problem, we aim to construct
solutions of the form
\[ u_1 = e^{\varphi_1} (u_1^\sharp + u_1^\flat), \qquad v_2 = e^{\varphi_2} (v_2^\sharp + v_2^\flat),\]
with $\varphi_1(t,x) + \overline{\varphi_2(t,x)} = 0$ 
for all $(t, x) \in \R \times \R^n$ for the equations 
\[ (i\partial_\tm + \Delta - V_1^{\rm ext}) u_1 = 0 \enspace \textnormal{and} \enspace  (i\partial_\tm + \Delta - \overline{V_2^{\rm ext}}) v_2 = 0 \enspace \textnormal{in} \enspace \R \times \R^n, \]
where
\begin{equation}
\label{def:extension}
 V_j^{\rm ext} (t, x) = \left\{ \begin{aligned}
& V_j (t, x) & & (t, x) \in \Sigma,\\
& 0 & & (t, x) \notin \Sigma.
\end{aligned} \right.
\end{equation}
The solutions $u_1$ and $v_2$ will be parametrized by 
$\nu \in \R^n$ with $|\nu| \geq \rho_0 $ so that
the real parts of $\varphi_1$ and $\varphi_2$ satisfy
$ \Re \varphi_1 = - \Re \varphi_2 = - \nu \cdot \x $,
and the inequalities
\[ \| u_1^\sharp \|_{L^2 (\Sigma^\nu_{< R})} + \| v_2^\sharp \|_{L^2 (\Sigma^\nu_{< R})} \lesssim (T R)^{1/2} \]
and
\begin{equation}
\label{in:remainder}
\| u_1^\flat \|_{L^2 (\Sigma^\nu_{< R})} + \| v_2^\flat \|_{L^2 (\Sigma^\nu_{< R})} \lesssim \frac{(T R)^{1/2 - \theta}}{|\nu|^{1/2}}
\end{equation}
hold for all $\nu \in \R^n$ with $|\nu| \geq \rho_0 > 0 $, where 
$\Sigma^\nu_{< R} = (0, T) \times \{ x \in \R^n : |x \cdot \nu| < |\nu| R \}$
and $\theta \in (0, 1/2)$. The functions $ u_1^\sharp $ and $ v_2^\sharp $ will
depend on $\hat{\nu} = \nu / |\nu| \in \Sph^{n - 1}$.
Thus, if plugging our CGO solutions 
into the orthogonality relation were allowed, we 
would conclude, by making $|\nu|$ go to infinity, that
\[\int_\Sigma (V_1 - V_2)u_1^\sharp \overline{v_2^\sharp} \, = 0 \qquad \forall \hat{\nu} \in \Sph^{n-1} .\]
Then, performing a density argument for all possible products 
$u_1^\sharp \overline{v_2^\sharp}$ we would prove that
$ V_1 (t, x) = V_2(t, x) $ for  almost every $(t, x) \in \Sigma$.

This approach presents two non-trivial challenges: first to construct the CGO 
solutions for the Schr\"odinger equation satisfying the above properties
---including that 
\[ \{ u_1^\sharp \overline{v_2^\sharp} : \hat{\nu} \in \Sph^{n-1} \} \]
is somehow  dense---, 
and then showing that the orthogonality relation is satisfied
for CGO solutions ---note that their restrictions to $\Sigma$ grow exponentially 
in certain regions of the space and, in consequence, they are not physical.

We will see in the \cref{sec:CGO_solutions} that in order to construct the CGO 
solutions we only need 
\begin{equation}
\label{cond:V_CGO}
\sum_{j \in \N_0} 2^{j(1/2 + \theta)} \| V_k \|_{L^\infty(\Sigma^\nu_j)} < \infty
\end{equation}
for all $\nu \in \R^n \setminus \{ 0 \}$ and $\theta \in (0, 1/2)$ as in the
inequality \eqref{in:remainder}.
Here and throughout the article $\N_0 = \{ 0 \} \cup \N$. Additionally, 
for $ \nu \in \R^n \setminus \{ 0 \}$, the sequence 
of strips $ \{ \Sigma^\nu_j : j \in \N_0 \} \subset \Sigma $ 
in the direction of $\nu$ are given by
\[ \Sigma^\nu_j = (0, T) \times \Omega^\nu_j\]
with
\[ \Omega^\nu_0 = \{ x \in \R^n : |x \cdot \nu| \leq |\nu|  \}, \quad
\Omega^\nu_j = \{ x \in \R^n : 2^{j - 1} |\nu| < |x \cdot \nu| \leq 2^j |\nu|  \} \enspace \forall j \in \N.\]

The condition \eqref{cond:V_CGO} 
is substantially weaker than having a super-exponential decay. The 
super-exponential decay is only used to extend the orthogonal 
identity from physical solutions to CGO solutions.
We have not analysed the necessity of this condition to prove the uniqueness
property of the initial-to-final-state inverse problem. However, we think that
it is interesting to study if the super-exponential decay could be 
replaced by a weaker exponential decay
(see \cite{zbMATH04103631,zbMATH00554321,zbMATH02208390}).

As we have mentioned earlier, in order to solve the data-prediction problem in
quantum mechanics seems crucial to derive a reconstruction algorithm to compute 
$V$ from its corresponding initial-to-final-state map $\mathcal{U}_T$.
This question does not seem trivial in the light of our analysis:
on the one hand, CGO solutions conform an essential tool, while on the other hand, 
they cannot be used as pieces of data because they are not physical.
Fortunately, part of our analysis to extend the orthogonality identity to CGO
solutions can be understood as a sort of density property, or approximation of
CGO by physical solutions.

It is known that solutions for the Schr\"odinger equation has
a unique continuation property whenever the initial and final states satisfy a 
condition known as the Hardy uncertainty principle (see 
\cite{zbMATH05127580,zbMATH05505931,zbMATH05365139,zbMATH05816157,zbMATH06022805}.)
We believe of capital importance to decide if the condition 
$\mathcal{U}_T^1 = \mathcal{U}_T^2$ can be replaced
by something on the spirit of the 
Hardy uncertainty principle for the difference of $\mathcal{U}_T^1 $ and 
$ \mathcal{U}_T^2$. Because of the property of weak unique continuation from the 
boundary, a weaker condition connected to the Hardy uncertainty principle
can be thought of as a counterpart of
the partial Dirichlet and Neumann data measured 
on the same open subset of the boundary for the Calderón problem 
(see, for example, 
the Calderón problem with local data in \cite{zbMATH06520428}).

\subsection{Contents}
The rest of the paper consists of other four sections. The 
\cref{sec:CGO_solutions} contains the construction of the CGO solutions.
One of the key points to construct these solutions is to prove the boundedness
of certain Fourier multiplier. This boundedness
is proved in the \cref{sec:boundedness_Fm}.
The \cref{sec:ItFmap} contains the proof of the orthogonality relation for 
physical solutions and also the extension to CGO solutions.
Finally, the \cref{sec:uniqueness} contains the proof of the \cref{th:uniqueness}.

\section{Complex geometrical optics solutions} \label{sec:CGO_solutions}
In order to motivate our choice of CGO solutions for the Schr\"odinger equation
with a Hamiltonian $-\Delta + V$, 
we start by discussing the set of exponential solutions solving the 
free Schr\"odinger equation
\begin{equation}
\label{eq:free-Schroedinger_exponential-solutions}
(i\partial_\tm + \Delta) e^\varphi = 0 \enspace \textnormal{in} \enspace \R \times \R^n.
\end{equation}
A simple computation shows that \eqref{eq:free-Schroedinger_exponential-solutions} holds if and only if
\begin{equation}
\label{eq:nonlinear_phase}
i\partial_\tm \varphi + \nabla \varphi \cdot \nabla \varphi + \Delta \varphi = 0 \enspace \textnormal{in} \enspace \R \times \R^n.
\end{equation}
To simplify the search of solutions of the previous equation, we restrict to 
linear phases:
\[\varphi(t,x) = \gamma\, t + \zeta \cdot x \qquad \forall (t, x) \in \R \times \R^n\]
with $\gamma \in \C$ and $\zeta \in \C^n$. The equation 
\eqref{eq:nonlinear_phase} is equivalent to
\[ i \gamma + \zeta \cdot \zeta = 0 \]
for $\varphi$ linear, or equivalently,
\[\Im \gamma = |\Re \zeta|^2 - |\Im \zeta|^2, \qquad \Re \gamma = - 2 \Re \zeta \cdot \Im \zeta. \]
Among the exponential solutions of the free Schr\"odinger equation with linear
phases, we are going to select two types of solutions: the first type is
\[ (t, x) \in \R \times \R^n \mapsto e^{-i|\kappa|^2 t + i \kappa \cdot x} \]
with $\kappa \in \R^n$, while the second type is
\[ (t, x) \in \R \times \R^n \mapsto e^{i|\nu|^2 t + \nu \cdot x} \]
with $\nu \in \R^n \setminus \{ 0 \}$. Note that the first type of solutions 
is unimodular and belongs to the
tempered distributions $\mathcal{S}^\prime(\R \times \R^n)$,
while the second type has certain exponential 
growth and only belongs to $\mathcal{D}^\prime(\R \times \R^n)$.
Let us explain the implications that
these different behaviours have for our purposes.

The product of solutions of the first type are \emph{dense}:
If $F \in L^1(\R \times \R^n)$ satisfies that
\begin{equation}
\label{id:orthogonality_1st}
\int_{\R \times \R^n} F (t, x) \, e^{-i|\kappa_1|^2 t + i \kappa_1 \cdot x} \, \overline{e^{-i|\kappa_2|^2 t + i \kappa_2 \cdot x}} \, \dd (t, x) = 0 \qquad \forall \kappa_1, \kappa_2 \in \R^n,
\end{equation}
then $F (t, x) = 0$ for almost every $(t, x) \in \R \times \R^n$.

Indeed, start by noticing that \eqref{id:orthogonality_1st} is equivalent to
state that the Fourier transform of $F$ satisfies
\begin{equation}
\label{id:Fourier_F}
\widehat{F} (|\kappa_1|^2 - |\kappa_2|^2, \kappa_2 - \kappa_1) = 0 \qquad \forall \kappa_1, \kappa_2 \in \R^n.
\end{equation}
Given $(\tau, \xi) \in \R \times \R^n$ so that $\xi \neq 0$, we choose
\begin{equation}
\label{def:eta1eta2}
\kappa_1 = -\frac{1}{2} \Big( 1 + \frac{\tau}{|\xi|^2} \Big) \xi, \qquad \kappa_2 = \frac{1}{2} \Big( 1 - \frac{\tau}{|\xi|^2} \Big) \xi,
\end{equation}
to obtain that $\widehat{F}(\tau, \xi) = 0$. Additionally,
if $(\tau, \xi) = (0, 0)$ we choose $\kappa_1 = \kappa_2 = 0$, and we obtain that
$\widehat{F}(0, 0) = 0$. Consequently, $\widehat{F}$ vanishes in 
$\{ (\tau, \xi) \in \R \times \R^n : \xi \neq 0 \} \cup \{ (0, 0) \}$.
Since the $\widehat{F}$ is 
continuous in $\R \times \R^n$ we can conclude that $\widehat{F}(\tau, \xi) = 0$ 
for all $(\tau, \xi) \in \R \times \R^n$. Since the Fourier transform is 
injective, we have that $F (t, x) = 0$ for a.e. $(t, x) \in \R \times \R^n$.

This density property also holds for the product of
physical solutions of the free Schr\"odinger equation
\[ \{ u \in C (\R ; L^2(\R^n)) : u = e^{i \tm \Delta} f, \enspace f \in L^2(\R^n) \} \]
since these are superpositions of solutions of the first type. 
More precisely, if $F \in L^1(\R \times \R^n) \cap L^1(\R ; L^\infty (\R^n))$
satisfies that
\begin{equation}
\label{id:orthogonality_2nd}
\int_{\R \times \R^n} F \, e^{i \tm \Delta} f_1 \, \overline{e^{i \tm \Delta} f_2} \, = 0 \qquad \forall f_1, f_2 \in L^2(\R^n),
\end{equation}
then $F (t, x) = 0$ for a.e. $(t, x) \in \R \times \R^n$.

Indeed, we can write
\begin{equation*}
\int_{\R \times \R^n} F \, e^{i \tm \Delta} f_1 \, \overline{e^{i \tm \Delta} f_2} = \frac{1}{(2\pi)^\frac{n-1}{2}} \int_{\R^n \times \R^n} 
\widehat{f_1}(\xi) \overline{\widehat{f_2}(\eta)} 
\widehat{F} (|\xi|^2 - |\eta|^2, \xi - \eta) \, \dd (\xi, \eta).
\end{equation*}
Then, by the fact that
\[\{ g_1 \otimes g_2 : g_j \in L^2(\R^n) \}\]
is dense in $L^1 (\R^n \times \R^n)$, 
the identity \eqref{id:orthogonality_2nd} is equivalent to
\[\widehat{F} (|\xi|^2 - |\eta|^2, \eta - \xi) = 0 \qquad \forall \xi, \eta \in \R^n.\]
This is the exact same condition as \eqref{id:Fourier_F}. Hence, $F (t, x) = 0$ 
for a.e. $(t, x) \in \R \times \R^n$. In the argument above we have used the 
notation $g_1 \otimes g_2 (x, y) = g_1 (x) g_2(y)$ for all 
$(x, y) \in \R^n \times \R^n$.

After these observations, it seems natural to consider the CGO solutions 
in the form
\[u = e^{i \tm \Delta} f + v, \]
where $v$ is a correction to make $u$ solution for the Schr\"odinger equation
with Hamiltonian $-\Delta + V$. However, in order to apply the density of 
product of physical solutions, we need to be able to ensure that the 
contribution of the correction term $v$ 
with respect to that of $e^{i \tm \Delta} f$ is negligible.
To do so, we introduce the exponential scale of the second type of solutions
since they will make the correction term to be negligible. The intuition behind
this claim is the following: Let us write
\begin{equation}
\label{def:complex_phase}
\varphi(t,x) = i |\nu|^2 t + \nu \cdot x \qquad \forall (t, x) \in \R \times \R^n
\end{equation}
with $\nu \in \R^n \setminus \{ 0 \}$ and note that
\[ e^{-\varphi} (i\partial_\tm + \Delta  ) [ e^\varphi w] = (i\partial_\tm + \Delta + 2 \nu \cdot \nabla) w \]
for every $w \in \mathcal{D}(\R \times \R^n)$. Furthermore, by Plancherel's 
theorem we have that
\[ \| e^{-\varphi} (i\partial_\tm + \Delta) [ e^\varphi w] \|_{L^2(\R \times \R^n)}^2 = \| (i\partial_\tm + \Delta) w \|_{L^2(\R \times \R^n)}^2 + 4 \| \nu \cdot \nabla w \|_{L^2(\R \times \R^n)}^2. \]
Recall that Poincar\'e's inequality tells that there exists an absolute constant
$C > 0$ so that
\[ |\nu| \| w \|_{L^2(\R \times \R^n)} \leq C R \| \nu \cdot \nabla w \|_{L^2(\R \times \R^n)} \]
whenever
\begin{equation}
\label{cond:support_w}
\supp w \subset \R \times \{ x \in \R^n : |\nu \cdot x| < |\nu| R \}.
\end{equation}
Thus, the exponential scale introduced by $e^\varphi$ ensures that
there exists an absolute constant
$C > 0$ so that
\[ |\nu| \| w \|_{L^2(\R \times \R^n)} \leq C R \| (i\partial_\tm + \Delta + 2 \nu \cdot \nabla) w \|_{L^2(\R \times \R^n)} \]
whenever the condition \eqref{cond:support_w} holds. This yields a gain of
$|\nu|$ for $ \| w \|_{L^2(\R \times \R^n)} $.
The reason behind this gain comes from the fact that the exponential conjugation
$ e^{-\varphi} (i\partial_\tm + \Delta ) \circ e^\varphi  = i\partial_\tm + \Delta + 2 \nu \cdot \nabla $
gives as a result a perturbation of $i\partial_\tm + \Delta $ whose symbol has an
imaginary part that grows as $|\nu|$.
As we will see later, a gain in the same spirit 
allows us to construct a correction term
that becomes negligible when $|\nu|$ is large.

Now that we know that the second type of exponential solutions will be also 
useful, we need to make them go along with the physical solutions: Notice that
\[ e^{-\varphi} (i\partial_\tm + \Delta  ) [ e^\varphi e^{i \tm \Delta} f] = 0  \Leftrightarrow \nu \cdot \nabla (e^{i \tm \Delta} f) = 0. \]
Thus, the solutions of the free Schr\"odinger equation that we are going to 
consider are $e^\varphi u^\sharp$ with $\varphi$ as in \eqref{def:complex_phase} 
and
\begin{equation}
\label{def:amplitud}
u^\sharp (t, x) = [e^{i t \Delta^\prime} \psi] (Qe_1 \cdot x, \dots, Qe_{n - 1} \cdot x) \qquad \forall (t, x) \in \R \times \R^n,
\end{equation}
where $-\Delta^\prime$ is the free Hamiltonian in
$\R^{n - 1}$,  $\psi \in L^2(\R^{n - 1})$ and $Q \in \Orth(n)$ ---the orthogonal
group--- such that $\nu = |\nu| Qe_n$.
It is convenient to state that $u^\sharp$ is not a physical solution:
Note that it is constant in the direction of $\nu$, it satisfies
\[(i\partial_\tm + \Delta + 2 \nu \cdot \nabla) u^\sharp = 0 \enspace \textnormal{in} \enspace \R \times \R^n \]
and there exists an absolute constant $C > 0$ so that
\begin{equation}
\label{cond:usharp}
\| u^\sharp \|_{L^2 (\Upsilon_j \times \Omega^\nu_k)} \leq C 2^{(j + k) /2} \| \psi \|_{L^2(\R^{n - 1})},
\end{equation}
for all $j, k \in \N_0$, where
\[ \Upsilon_0 = \{ t \in \R : |t| \leq 1  \}, \quad
\Upsilon_j = \{ t \in \R : 2^{j - 1} < |t| \leq 2^j \} \enspace \forall j \in \N \]
and, recall from the introduction that
\[ \Omega^\nu_0 = \{ x \in \R^n : |x \cdot Qe_n| \leq 1  \},\quad
\Omega^\nu_k = \{ x \in \R^n : 2^{k - 1} < |x \cdot Qe_n| \leq 2^k \}, \enspace \forall k \in \N. \]

Now that we have a good choice for the solutions of the free Schr\"odinger 
equation, we will construct the CGO solutions for the equation
\begin{equation}
i\partial_\tm u = - \Delta u + \tilde{V} u \enspace \textnormal{in} \enspace \R \times \R^n,
\label{eq:schroedinger}
\end{equation}
in the form
\begin{equation}
\label{id:CGO}
u = e^\varphi (u^\sharp + u^\flat).
\end{equation}
Here $\tilde V : \R \times \R^n \rightarrow \C$ is defined as
\begin{equation}
\label{id:Vtilde}
\tilde V (t, x) = \left\{ \begin{aligned}
& V (t, x) & & (t, x) \in \Sigma,\\
& 0 & & (t, x) \notin \Sigma.
\end{aligned} \right.
\end{equation}
Choosing $\varphi$ as in \eqref{def:complex_phase} and $u^\sharp$ as in 
\eqref{def:amplitud}, we have that
\[ e^{-\varphi} \big(i\partial_\tm + \Delta - \tilde{V}  \big) \big[ e^\varphi (u^\sharp + u^\flat) \big] = (i\partial_\tm + \Delta + 2 \nabla \varphi \cdot \nabla - \tilde{V}) u^\flat - \tilde{V} u^\sharp. \]
Then, we choose the remainder $u^\flat$ solving
\begin{equation}
(i\partial_\tm + \Delta + 2 \nabla \varphi \cdot \nabla - \tilde{V}) u^\flat = \tilde{V} u^\sharp \enspace \textnormal{in} \enspace \R \times \R^n.
\label{eq:remainder-equation}
\end{equation}

In order to find $u^\flat$ solving the equation \eqref{eq:remainder-equation},
we are going to introduce a Fourier multiplier $S_\nu$  
such that if $u = S_\nu f $ with
$f \in \mathcal{S}(\R \times \R^n)$ then
\begin{equation}
\label{eq:constant_coefficients}
(i\partial_\tm + \Delta + 2 \nu \cdot \nabla) u = f \enspace \textnormal{in} \enspace \R \times \R^n.
\end{equation}
The idea is to construct $u^\flat$ as a solution for the equation
\begin{equation}
\label{eq:Neumann}
(\Id - S_\nu \circ M_{\tilde{V}} ) u^\flat =  S_\nu (\tilde{V} u^\sharp),
\end{equation}
where $M_{\tilde{V}} $ denotes the operator $ u \mapsto \tilde{V} u$, since 
$u^\flat$ satisfies the identity \eqref{eq:remainder-equation} if $u^\flat$ 
solves the equation \eqref{eq:Neumann}.
A way to solve \eqref{eq:Neumann} consists in finding 
a Banach space where the operator
$S_\nu \circ M_{\tilde{V}}$ becomes a contraction. In fact, we will find a 
couple of family of spaces $X_\nu^\theta$ and $Y_\nu^\theta$
with $\theta \in (0, 1/2)$ so that, if
$V \in L^\infty( \Sigma ) $ such that
\[ \sum_{j \in \N_0} 2^{j/2} \| V \|_{L^\infty(\Sigma^\nu_j)} < \infty, \]
with $\Sigma^\nu_j = (0, T) \times \Omega^\nu_j$, then
\begin{equation}
\label{in:to_construct_remainder}
\| S_\nu \|_{\mathcal{L}(X_\nu^\theta; Y_\nu^{1/2 - \theta})} + |\nu|^{1/2} \| M_{\tilde{V}} \|_{\mathcal{L}(Y_\nu^{1/2 - \theta}; X_\nu^\theta)} \lesssim 1
\end{equation}
as $|\nu|$ goes to infinity.\footnote{If $a$ and $b$ are positive constants, we 
write $a \lesssim b$ or equivalently $b \gtrsim a$, when there exists $C > 0$ so 
that $a \leq C b$. We refer to $C$ as the implicit constant. Additionally, if 
$a \lesssim b$ and $b \lesssim a$, we write $a \eqsim b$.}
As we will see, this is enough to define an inverse of the operator
$(\Id - S_\nu \circ M_{\tilde{V}})$ at least when $|\nu|$ is 
sufficiently large. Finally, in order to solve \eqref{eq:Neumann} we have to 
ensure that $\tilde{V}u^\sharp \in X_\nu^\theta$,
for that, we will require that for some $\theta \in (0, 1/2)$
\begin{equation}
\label{cond:Vepsilon}
\sum_{j \in \N_0} 2^{j(1/2 + \theta)} \| V \|_{L^\infty(\Sigma^\nu_j)} < \infty.
\end{equation}

\subsection{The Banach spaces}
Here we introduce the family of spaces $X_\nu^\theta$ and $Y_\nu^\theta$ with 
$\theta \in (0, 1/2)$ used in \eqref{in:to_construct_remainder}.
\begin{definition} \label{def:Xnu}
For $ \nu \in \R^n \setminus \{ 0 \}$, introduce the sequence 
of strips $ \{ \Pi^\nu_\alpha : \alpha \in \N_0^2 \} \subset \R \times \R^n $ 
in the direction of $\nu$ given by
\[ \Pi^\nu_\alpha = \Upsilon_{\alpha_1} \times \Omega^\nu_{\alpha_2}\]
with $\alpha = (\alpha_1, \alpha_2)$,
\begin{align*}
& \Upsilon_0 = \{ t \in \R : |t| \leq 1  \}, & &
\Upsilon_j = \{ t \in \R : 2^{j - 1} < |t| \leq 2^j \} \enspace \forall j \in \N; \\
& \Omega^\nu_0 = \{ x \in \R^n : |x \cdot \nu| \leq |\nu|  \}, & &
\Omega^\nu_j = \{ x \in \R^n : 2^{j - 1} |\nu| < |x \cdot \nu| \leq 2^j |\nu|  \} \enspace \forall j \in \N.
\end{align*}
Define $X_\nu^\theta $ with $\theta \in (0, 1/2)$ 
as the set of measurable functions 
$f:\R \times \R^n \rightarrow \C $ so that
\[ \sum_{\alpha \in \N_0^2} 2^{|\alpha| \theta} \| f \|_{L^2 (\Pi^\nu_\alpha)} < \infty, \]
and the functional 
$ \| \centerdot \|_{X_\nu^\theta} : X_\nu^\theta \rightarrow [0, \infty) $ as
\[ \| f \|_{X_\nu^\theta} = |\nu|^{-1/4} \sum_{\alpha \in \N_0^2} 2^{|\alpha| \theta} \| f \|_{L^2 (\Pi^\nu_\alpha)}. \]
\end{definition}
Redefining the set $X_\nu^\theta$ by identifying functions which 
are equal almost everywhere, 
we have that the functional $\| \centerdot \|_{X_\nu^\theta}$ becomes a 
norm and $(X_\nu^\theta, \| \centerdot \|_{X_\nu^\theta})$ is a Banach space.

\begin{definition} We define $Y_\nu^\theta $ as the set of measurable functions 
$u:\R \times \R^n \rightarrow \C $ so that
\[ \sup_{\alpha \in \N_0^2} \big[ 2^{-|\alpha|\theta} \| u \|_{L^2 (\Pi^\nu_\alpha)} \big] < \infty, \]
and the functional 
$ \| \centerdot \|_{Y_\nu^\theta} : Y_\nu^\theta \rightarrow [0, \infty) $ as
\[ \| u \|_{Y_\nu^\theta} = |\nu|^{1/4} \sup_{\alpha \in \N_0^2} \big[ 2^{-|\alpha|\theta} \| u \|_{L^2 (\Pi^\nu_\alpha)} \big]. \]
\end{definition}
Redefining now the set $Y_\nu^\theta$ by identifying functions which 
are equal almost everywhere, 
we have that the functional $\| \centerdot \|_{Y_\nu^\theta}$ becomes a 
norm, and $(Y_\nu^\theta, \| \centerdot \|_{Y_\nu^\theta})$ is a Banach space.

\begin{remark}\label{rm:duality} The space $(Y_\nu^\theta, \| \centerdot \|_{Y_\nu^\theta})$
can be identified with the dual of 
$(X_\nu^\theta, \| \centerdot \|_{X_\nu^\theta})$.
\end{remark}

Recall here that, if $V$ satisifies \eqref{cond:Vepsilon} for some $\theta$, we can use the spaces
$X_\nu^\theta$ and $Y_\nu^{1/2 - \theta}$ to solve the equation 
\eqref{eq:remainder-equation}.

\subsection{The construction of the remainder}
Here we explain and, to some extent, prove \eqref{in:to_construct_remainder}. 
Then, we show how to solve \eqref{eq:Neumann} and \eqref{eq:remainder-equation}.
Finally, we state in a theorem the existence of the CGO solutions and their
properties.

\begin{lemma}\label{lem:bilinearV}\sl Consider $\nu \in \R^n \setminus \{0\}$ and $V \in L^\infty ( \R \times \R^n) $. Then,
\[ \int_{\R \times \R^n} | V u \overline{v} | \, \leq \frac{1}{|\nu|^{1/2}} \Big( \sum_{\alpha \in \N_0^2} 2^{|\alpha|/2} \| V \|_{L^\infty (\Pi^\nu_\alpha)} \Big) \| u \|_{Y_\nu^{1/2 -\theta}} \| v \|_{Y_\nu^\theta} \]
for all $u \in Y_\nu^{1/2 -\theta} $ and $v \in Y_\nu^\theta $.
\end{lemma} 

\begin{proof}
It is a simple consequence of the Cauchy--Schwarz inequality
\begin{align*}
& \int_{\R \times \R^n} | V u \overline{v} | \, \leq \sum_{\alpha \in \N_0^2} \| V \|_{L^\infty (\Pi^\nu_\alpha)} \| u \|_{L^2 (\Pi^\nu_\alpha)} \| v \|_{L^2 (\Pi^\nu_\alpha)} \\
& \leq  \Big( \sum_{\alpha \in \N_0^2} 2^{|\alpha|/2} \| V \|_{L^\infty (\Pi^\nu_\alpha)} \Big) \sup_{\alpha \in \N_0^2} \big[ 2^{-|\alpha|(1/2 - \theta)} \| v \|_{L^2 (\Pi^\nu_\alpha)} \big] \sup_{\alpha \in \N_0^2} \big[ 2^{-|\alpha|\theta} \| v \|_{L^2 (\Pi^\nu_\alpha)} \big].
\end{align*}
Writing now the norms with the corresponding powers of $|\nu|^{1/4}$, we obtain
the inequality stated.
\end{proof}

The \cref{lem:bilinearV} and the \cref{rm:duality} imply the
following corollary.

\begin{corollary}\label{cor:smallness} \sl Consider $\nu \in \R^n \setminus \{0\}$ and
$V \in L^\infty ( \R \times \R^n) $ such that
\[\sum_{\alpha \in \N_0^2} 2^{|\alpha|/2} \| V \|_{L^\infty (\Pi^\nu_\alpha)} < \infty \]
Then, the operator
\[M_V : u\in Y_\nu^{1/2 - \theta} \mapsto Vu \in X_\nu^\theta\]
has a norm bounded as follows
\[ \| M_V \|_{\mathcal{L} (Y_\nu^{1/2 - \theta}; X_\nu^\theta)} \leq |\nu|^{-1/2} \sum_{\alpha \in \N_0^2} 2^{|\alpha|/2} \| V \|_{L^\infty (\Pi^\nu_\alpha)}. \]
\end{corollary}

In the \cref{sec:boundedness_Fm} we will define a multiplier 
$S_{(\lambda, \zeta)}$, with $\lambda \in \R$ and $\zeta \in \C^n$ such that
$\Re \zeta \neq 0$, and we will prove its boundedness.
In the case $(\lambda, \zeta) = (|\nu|^2, \nu)$ with 
$\nu \in \R^n \setminus \{0  \}$, this multiplier satisfies 
\eqref{eq:constant_coefficients} and it is denoted by $S_\nu$ for simplicity.
Here we just state the boundedness for 
$S_\nu = S_{(|\nu|^2, \nu)}$.

\begin{proposition}\label{th:boundedness_multiplier} \sl The Fourier multiplier 
$S_\nu$ can be 
extended to a bounded operator, still denoted by $S_\nu$, 
from $X_\nu^\theta$ to $Y_\nu^{1/2 - \theta}$ for every $\theta \in (0, 1/2)$.
Moreover, there exists an absolute constant $C> 0$ so that 
\[\| S_\nu \|_{\mathcal{L} (X_\nu^\theta; Y_\nu^{1/2 - \theta})} \leq C\]
for all $\nu \in \R^n \setminus \{ 0 \}$.
\end{proposition}

\begin{corollary}\label{cor:inverse} \sl 
Consider $V \in L^\infty ( \R \times \R^n) $ such that
\[\sum_{\alpha \in \N_0^2} 2^{|\alpha|/2} \| V \|_{L^\infty (\Pi^\nu_\alpha)} < \infty \]
for all $\nu \in \R^n \setminus \{ 0 \}$.
Then, there exists a constant $\rho_0> 0$ that only
depends on $V$ so that the operator
\[ (\Id - S_\nu \circ M_V) : Y_\nu^{1/2 - \theta} \rightarrow Y_\nu^{1/2 - \theta} \]
has a bounded inverse in $Y_\nu^{1/2 - \theta}$ for all 
$\nu \in \R^n \setminus \{ 0 \}$ such that
$|\nu| \geq \rho_0$ and all $\theta \in (0, 1/2)$. Here $\Id$ denotes the identity 
map.
\end{corollary}

\begin{proof} Using the \cref{th:boundedness_multiplier} and 
\cref{cor:smallness} the conclusion of this corollary follows directly from the
Neumann-series theorem.
\end{proof}

With these results at hand, we are in position to solve the equation \eqref{eq:Neumann} 
and consequently \eqref{eq:remainder-equation}.
\begin{proposition}\sl 
Consider $V \in L^\infty ( \R \times \R^n) $ such that
\[\sum_{\alpha \in \N_0^2} 2^{|\alpha|/2} \| V \|_{L^\infty (\Pi^\nu_\alpha)} < \infty \]
for all $\nu \in \R^n \setminus \{ 0 \}$.
There exists a constant $\rho_0> 0$ that only
depends on $V$ so that, for every $\nu \in \R^n$ 
such that $|\nu| \geq \rho_0$ and every $f \in X_\nu^\theta$ with 
$\theta \in (0, 1/2)$, the function
\[ u = (\Id - S_\nu \circ M_V)^{-1} (S_\nu  f) \in Y_\nu^{1/2 - \theta}, \]
solves the equation
\[(i\partial_\tm + \Delta + 2 \nu \cdot \nabla - V) u = f \enspace \textnormal{in} \, \R \times \R^n.\]
Moreover, there exists an absolute constant $C>0$ so that
\[ \| u \|_{Y_\nu^{1/2 - \theta}} \leq C \| f \|_{X_\nu^\theta}.\]
\end{proposition}

\begin{proof}
For $\nu \in \R^n$ such that $|\nu| \geq \rho_0$
and $f \in X_\nu^\theta$ with $\theta \in (0, 1/2)$,
the \cref{cor:inverse} and \cref{th:boundedness_multiplier} implies that
\[ \| u \|_{Y_\nu^{1/2 - \theta}}  \lesssim \| f \|_{X_\nu^\theta}. \]

By definition, we have that
\[(\Id - S_\nu \circ M_V) u = S_\nu  f.\]
Thus, applying the differential operator 
$i\partial_\tm + \Delta + 2 \nu \cdot \nabla$ 
to this identity we have that
\[(i\partial_\tm + \Delta + 2 \nu \cdot \nabla - V) u = f \enspace \textnormal{in} \, \R \times \R^n.\]
This concludes the proof.
\end{proof}

Recall that in order to construct the remainder term of the CGO solution
$u = e^{\varphi} (u^\sharp + u^\flat)$, that is
$u^\flat$ solving \eqref{eq:Neumann}, we have to 
ensure that $\tilde V u^\sharp \in X_\nu^\theta$, which is satisfied by
\eqref{cond:usharp} and \eqref{cond:Vepsilon}:
\[ \sum_{\alpha \in \N_0^2} 2^{|\alpha| \theta} \| \tilde{V} u^\sharp \|_{L^2(\Pi^\nu_\alpha)} \lesssim T^{\theta + 1/2} \| \psi \|_{L^2(\R^{n - 1})} 
\sum_{j \in \N_0} 2^{j(\theta + 1/2)} \| V \|_{L^\infty(\Sigma^\nu_j)}. \]

We finish this section by stating the theorem which establishes the existence 
of the CGO solutions and their properties.

\begin{theorem}\label{th:CGO}\sl
Consider $V \in L^\infty (\Sigma)$ such that
for some $\theta \in (0, 1/2)$ we have
\[\sum_{j \in \N_0} 2^{j(1/2 + \theta)} \| V \|_{L^\infty(\Sigma^\nu_j)} < \infty \]
for all $\nu \in \R^n \setminus \{ 0 \}$.
Let $\tilde V$ denote the trivial extension given in \eqref{id:Vtilde}.
For $\nu \in \R^n \setminus \{0\}$, define
\[\varphi(t,x) = i |\nu|^2 t + \nu \cdot x \qquad \forall (t, x) \in \R \times \R^n.\]
Additionally, for $\psi \in L^2(\R^{n - 1})$, define
\[ u^\sharp (t, x) = [e^{i t \Delta^\prime} \psi] (Qe_1 \cdot x, \dots, Qe_{n - 1} \cdot x) \qquad \forall (t, x) \in \R \times \R^n,\]
where $\Delta^\prime$ is the free Hamiltonian in
$\R^{n - 1}$, and $Q \in \Orth(n)$ such that $\nu = |\nu| Qe_n$.
Then, there exist positive constants $\rho_0$ and $C$ that only
depend on $V$ and $T$ so that, for every $\nu \in \R^n \setminus \{ 0 \}$
such that $|\nu| \geq \rho_0$ and every $\psi \in L^2(\R^{n - 1})$
there is $u^\flat \in Y_\nu^{1/2 - \theta}$ so that
\[ u = e^\varphi (u^\sharp + u^\flat) \]
solves the equation
\begin{equation}
\label{eq:CGO}
i\partial_\tm u = - \Delta u + \tilde V u \enspace \textnormal{in} \enspace \tilde \R \times \R^n
\end{equation}
and
\[ \sup_{\alpha \in \N_0^2} \big[ 2^{-|\alpha|/2} \| u^\sharp \|_{L^2 (\Pi^\nu_\alpha)} \big] + |\nu|^{1/2} \sup_{\alpha \in \N_0^2} \big[ 2^{-|\alpha|(1/2 - \theta)} \| u^\flat \|_{L^2 (\Pi^\nu_\alpha)} \big] \leq C \| \psi \|_{L^2(\R^{n - 1})}. \]
\end{theorem}

\section{The boundedness of $S_{(\lambda, \zeta)}$}\label{sec:boundedness_Fm}

In this section we prove a generalization of the 
\cref{th:boundedness_multiplier}. To do so,
let us focus on inverting the constant-coefficient differential operator 
\begin{equation}
\label{term:constant-coefficient_operator}
i\partial_\tm + \Delta + 2 \zeta \cdot \nabla + \zeta \cdot \zeta - \lambda
\end{equation}
and finding bounds for its inverse on the spaces 
$X_{\Re \zeta}^\theta$ and $Y_{\Re \zeta}^{1/2 - \theta}$ with 
$\theta \in (0, 1/2)$.
Here $\lambda \in \R$ and $\zeta \in \C^n$ do not necessarily satisfy 
the condition $\lambda = \zeta \cdot \zeta$. The symbol of 
\eqref{term:constant-coefficient_operator} is the polynomial function
\begin{equation}
\label{id:symbol}
p_{(\lambda, \zeta)} (\tau, \xi) = - \lambda + \zeta \cdot \zeta - \tau - |\xi|^2 + i2 \zeta \cdot \xi \qquad \forall (\tau, \xi) \in \R \times \R^n.
\end{equation}
For $\lambda \in \R$ and $\zeta \in \C^n$ such that $\Re \zeta \neq 0$, set 
\[\Gamma_{(\lambda, \zeta)} = \{ (\tau, \xi) \in \R \times \R^n : p_{(\lambda, \zeta)} (\tau, \xi) = 0 \}.\]
Then, the function 
\[(\tau, \xi) \in (\R \times \R^n) \setminus \Gamma_{(\lambda, \zeta)} \longmapsto \frac{1}{p_{(\lambda, \zeta)} (\tau, \xi)} \in \C\]
can be extended to $\R \times \R^n$ as a locally integrable function. Thus, for 
$f \in \mathcal{S}(\R \times \R^n)$, we define
\begin{equation}
\label{id:multiplier_definition}
S_{(\lambda, \zeta)} f (t, x) = \frac{1}{(2\pi)^\frac{n+1}{2}} \int_{\R \times \R^n} e^{i (t \tau + x \cdot \xi)} \frac{1}{p_{(\lambda, \zeta)} (\tau, \xi)} \widehat{f}(\tau, \xi) \, \dd (\tau, \xi)
\end{equation}
for all $(t, x) \in \R \times \R^n$.

The goal of this section is to show that 
there exists an absolute constant $C > 0$ such that
\begin{equation}
\label{in:lambda-zeta}
\| S_{(\lambda, \zeta)} f \|_{L^2 (\Pi^{\Re \zeta}_\alpha)} \leq \frac{C}{|\Re \zeta|^{1/2}} 2^{|\alpha| \theta} \sum_{\beta \in \N_0^2} 2^{|\beta| (1/2 - \theta)} \| f \|_{L^2 (\Pi^{\Re \zeta}_\beta)} 
\end{equation}
for all $f \in \mathcal{S} (\R \times \R^n)$, $\alpha \in \N_0^2 $,
$\lambda \in \R$, $\zeta \in \C^n$ with $\Re \zeta \neq 0$, and 
$\theta \in (0, 1/2)$.

The \cref{th:boundedness_multiplier} follows from the inequality 
\eqref{in:lambda-zeta} setting 
$(\lambda, \zeta) = (|\nu|^2, \nu) $ for $\nu \in \R^n \setminus \{ 0 \}$, 
and taking supremum in $\alpha \in \N_0^2$.

The fact that
\[ p_{(\lambda, \zeta)} (\tau, \xi) = p_{(0, \Re\zeta)} (\tau + \lambda, \xi + \Im \zeta) \qquad \forall (\tau, \xi) \in \R \times \R^n,\]
implies that \eqref{in:lambda-zeta} is equivalent to prove that
there exits an absolute constant $C > 0$ such that
\begin{equation}
\label{in:nu}
\| S_{(0, \nu)} f \|_{L^2 (\Pi^\nu_\alpha)} \leq \frac{C}{|\nu|^{1/2}} 2^{|\alpha| \theta} \sum_{\beta \in \N_0^2} 2^{|\beta| (1/2 - \theta)} \| f \|_{L^2 (\Pi^\nu_\beta)} 
\end{equation}
for all $f \in \mathcal{S} (\R \times \R^n)$, $\alpha \in \N_0^2 $,
$\nu \in \R^n \setminus \{ 0 \}$, and $\theta \in (0, 1/2)$.
 
Since the operator $S_{(0, \nu)}$ has certain \emph{homogeneity}, instead of
studying the viability of \eqref{in:nu}, we will analyse a \emph{normalized} 
version of it. Consider the polynomial function
\[ p(\tau, \xi) = \tau - |\xi|^2 + i\xi_n \qquad \forall\, (\tau, \xi) \in \R \times \R^n, \]
where $\xi_n = \xi \cdot e_n$ and $\{e_1,\ldots,e_n\}$ is
the standard basis of $\R^n$. Setting
\[\Gamma = \{ (\tau, \xi) \in \R \times \R^n : p (\tau, \xi) = 0 \}.\]
Then, the function 
\[(\tau, \xi) \in (\R \times \R^n) \setminus \Gamma \longmapsto \frac{1}{p (\tau, \xi)} \in \C\]
can be extended to $\R \times \R^n$ as a locally integrable function. Thus, for 
$f \in \mathcal{S}(\R \times \R^n)$, we define
\begin{equation}
\label{def:S}
S f (t, x) = \frac{1}{(2\pi)^\frac{n+1}{2}} \int_{\R \times \R^n} e^{i (t \tau + x \cdot \xi)} \frac{1}{p (\tau, \xi)} \widehat{f}(\tau, \xi) \, \dd (\tau, \xi) \quad \forall (t, x) \in \R \times \R^n.
\end{equation}
In order to relate the operators $S$ and $S_{(0, \nu)}$, we choose 
$Q \in \Orth(n)$ so that $\nu = |\nu| Q e_n$ and make the change
$ \sigma = |\nu|^2 (1 - 4 \tau)$ and $ \eta = 2 |\nu| Q \xi$ that implies
\[ p_{(0, \nu)} (\sigma, \eta) = 4 |\nu|^2 p(\tau, \xi). \]
One can check that if $f, g \in \mathcal{S}(\R \times \R^n)$ satisfy
\[ \widehat{g}(\tau, \xi) = (2 |\nu|)^{n+2} \widehat{f} ( |\nu|^2 - 4 |\nu|^2 \tau, 2 |\nu| Q \xi) \qquad \forall (\tau, \xi) \in \R \times \R^n, \]
then
\begin{align*}
& f (s, y) = e^{i|\nu|^2 s} g (-4 |\nu|^2 s, 2 |\nu| Q^\T y),\\ 
& S_{(0, \nu)} f (s, y) = \frac{e^{i|\nu|^2 s}}{4|\nu|^2} S g (-4 |\nu|^2 s, 2 |\nu| Q^\T y).
\end{align*}
After these relations it is a simple task to check that
\begin{equation}
\label{id:relations}
\begin{aligned}
& \| f \|_{L^2(\Pi^\nu_\alpha)} = (2 |\nu|)^{-(n + 2)/2} \| g \|_{L^2(I_{\alpha_1} \times \R^{n - 1} \times J_{\alpha_2})},\\ 
& \| S_{(0, \nu)} f \|_{L^2(\Pi^\nu_\alpha)} = \frac{(2 |\nu|)^{-(n + 2)/2}}{4|\nu|^2} \| S g \|_{L^2(I_{\alpha_1} \times \R^{n - 1} \times J_{\alpha_2})}
\end{aligned}
\end{equation}
where $I_0 = [-4 |\nu|^2, 4 |\nu|^2] $, $J_0 = [-2 |\nu|, 2 |\nu|]$, and
$ I_k = [-2^{k +2} |\nu|^2, - 2^{k + 1} |\nu|^2 ) \cup (2^{k+1} |\nu|^2,  2^{k +2} |\nu|^2] $, 
$ J_k = [-2^{k +1} |\nu|, - 2^k |\nu| ) \cup (2^k |\nu|,  2^{k +1} |\nu|] $
for $k \in \N$.

Before stating the inequality for the Fourier multiplier $S$ that implies 
\eqref{in:nu}, we need to agree a definition. 

\begin{definition}
We say that a measurable subset $E$ of $\R^m$ with $m \in \N$ is regular if
\[0 < |\mathring E| = |\overline{E}| < \infty.\] 
Here $\mathring E $ and $\overline{E}$ denote the interior and closure of $E$, while 
$|F|$ denotes the measure of a measurable set $F$.
\end{definition}

\begin{lemma}\label{lem:boundedness_S} \sl
There exists an absolute constant $C > 0$
such that, for every $\theta \in (0, 1/2)$,
every measurable subsets $I$ and $J$ of $\R$ and every partition
$\{ I_N : N \in \N_0 \}$ and $\{ J_N : N \in \N_0 \}$ of $\R$ 
compounded by regular measurable sets,
we have that  
\[ \| S f \|_{L^2(I \times \R^{n - 1} \times J)} \leq C |I|^\theta |J|^\theta \sum_{\alpha \in \N_0^2} |I_{\alpha_1}|^{1/2 - \theta} |J_{\alpha_2}|^{1/2 - \theta} \| f \|_{L^2(I_{\alpha_1} \times \R^{n - 1} \times J_{\alpha_2})} \]
for $f \in \mathcal{S}(\R \times \R^n)$ and $\alpha = (\alpha_1, \alpha_2)$.
\end{lemma}

Before proving the \cref{lem:boundedness_S}, we will see how this
implies the inequality \eqref{in:nu}. For $I_{\alpha_1}$ and $J_{\alpha_2}$ with 
$\alpha \in \N_0^2$ as in the identities in \eqref{id:relations} we have
\[\| S g \|_{L^2(I_{\alpha_1} \times \R^{n - 1} \times J_{\alpha_2})} \lesssim  |\nu|^{1+1/2} 2^{|\alpha| \theta} \sum_{\beta \in \N_0^2} 2^{|\beta|(1/2 - \theta)} \| g \|_{L^2(I_{\beta_1} \times \R^{n - 1} \times J_{\beta_2})}.  \]
Then, the identities in \eqref{id:relations} yield
the inequality \eqref{in:nu}.

\begin{proof}[Proof of the \cref{lem:boundedness_S}]
Since the partitions $\{ I_N : N \in \N_0 \}$ and $\{ J_N : N \in \N_0 \}$
are compound by regular measurable 
sets, for every $f \in \mathcal{S}(\R \times \R^n)$ and every $\varepsilon > 0$, there exists a sequence $ \{ f_N : N \in \N \} \subset \mathcal{S}(\R \times \R^n) $ 
such that
\[\supp f_N \subset \bigcup_{\alpha \in \N_0^2} (\mathring{I_{\alpha_1}} \times \R^{n - 1} \times \mathring{J_{\alpha_2}}) \qquad \forall N \in \N, \]
and an $N_0 \in \N$
\[\sum_{\alpha \in \N_0^2} |I_{\alpha_1}|^{1/2 - \theta} |J_{\alpha_2}|^{1/2 - \theta} \| f - f_N \|_{L^2(I_{\alpha_1} \times \R^{n - 1} \times J_{\alpha_2})} < \varepsilon \qquad \forall N \geq N_0.\]
By density, we can assume that $f \in \mathcal{S}(\R \times \R^n)$ and
\begin{equation}
\label{cond:support}
\supp f \subset \bigcup_{\alpha \in \N_0^2} (\mathring{I_{\alpha_1}} \times \R^{n - 1} \times \mathring{J_{\alpha_2}}). 
\end{equation}
Write $f = \sum_{\alpha \in \N_0^2} f_\alpha $ with $f_\alpha \in \mathcal{S}(\R^n)$ denoting the function
\[ f_\alpha (t, x^\prime, x_n) = \mathbf{1}_{I_{\alpha_1}}(t) \mathbf{1}_{J_{\alpha_2}} (x_n) f (t, x^\prime, x_n) \qquad \forall (t, x^\prime, x_n) \in \R \times \R^{n - 1} \times \R, \]
where $\mathbf{1}_{I_{\alpha_1}}$ and $ \mathbf{1}_{J_{\alpha_2}}$ denote the characteristic functions of the sets $I_{\alpha_1}$ and $J_{\alpha_2}$.
Then, 
\begin{equation}
\label{in:triangle}
\| S f \|_{L^2(I \times \R^{n - 1} \times J)} \leq \sum_{\alpha \in \N_0^2} \| S f_\alpha \|_{L^2(I \times \R^{n - 1} \times J)}.
\end{equation}
Our goal is to prove that
\begin{equation}
\label{in:goal}
\| S f_\alpha \|_{L^2(I \times \R^{n - 1} \times J)} \leq C |I|^\theta|J|^\theta |I_{\alpha_1}|^{1/2 - \theta} |J_{\alpha_2}|^{1/2 - \theta} \| f_\alpha \|_{L^2(\R \times \R^n)},
\end{equation}
since the identity 
$ \| f_\alpha \|_{L^2(\R \times \R^n)} = \| f \|_{L^2(I_{\alpha_1} \times \R^{n - 1} \times J_{\alpha_2})}$
together with the inequalities \eqref{in:triangle} and 
\eqref{in:goal} would imply the inequality in the statement for 
$f \in \mathcal{S}(\R \times \R^n)$ satisfying \eqref{cond:support}. 

A simple computation shows that
\[ S f_\alpha (t, x^\prime, x_n) = \frac{1}{(2\pi)^\frac{n-1}{2}} \int_{\R^{n-1}} e^{i x^\prime \cdot \xi^\prime} e^{i t |\xi^\prime|^2} T[ g_\alpha(\centerdot; \xi^\prime)] (t, x_n) \, \dd \xi^\prime \]
with $T$ the Fourier multiplier in $\R^2$ defined in the \cref{lem:multiplierT}
below, and
\[ g_\alpha(y_1, y_2; \xi^\prime) = e^{-i y_1 |\xi^\prime|^2} \mathcal{F} [f_\alpha(y_1, \centerdot, y_2)] (\xi^\prime) \qquad \forall (y_1, y_2) \in \R \times \R. \]
Here $ \mathcal{F} [f_\alpha(y_1, \centerdot, y_2)] $ denotes the Fourier 
transform of the function $ x^\prime \in \R^{n - 1} \mapsto f_\alpha(y_1, x^\prime, y_2) $.
Then, by Plancherel's theorem in $\R^{n - 1}$, we have that
\[ \| S f_\alpha (t, \centerdot, x_n) \|_{L^2 (\R^{n - 1})}^2 = \int_{\R^{n-1}} |T[ g_\alpha(\centerdot; \xi^\prime)] (t, x_n)|^2 \, \dd \xi^\prime. \]
In the \cref{lem:multiplierT} we will state an inequality that we use 
in \eqref{in:usingT} after 
applying the previous identity and changing the order of integration:
\begin{equation}
\label{in:usingT}
\begin{aligned}
\| S f_\alpha \|_{L^2(I \times \R^{n - 1} \times J)}^2 &= \int_{\R^{n-1}} \|T[ g_\alpha(\centerdot; \xi^\prime)] \|_{L^2(I \times J)}^2 \, \dd \xi^\prime \\
& \lesssim (|I| |J|)^{2\theta} (|I_{\alpha_1}| |J_{\alpha_2}|)^{1 - 2\theta} \int_{\R^{n-1}} \|g_\alpha(\centerdot; \xi^\prime)] \|_{L^2(\R^2)}^2  \, \dd \xi^\prime
\end{aligned}
\end{equation}
since $\supp g_\alpha(\centerdot; \xi^\prime) \subset \mathring{I_{\alpha_1}} \times \mathring{J_{\alpha_2}}$
for all $\xi^\prime \in \R^{n - 1}$. Note that
\begin{equation}
\label{id:back}
\begin{aligned}
\int_{\R^{n-1}} \|g_\alpha(\centerdot; \xi^\prime)] \|_{L^2(\R^2)}^2  \, \dd \xi^\prime &= \int_{\R \times \R} \| \mathcal{F} [f_\alpha(t, \centerdot, x_n)] \|_{L^2(\R^{n - 1})}^2 \dd (t, x_n) \\
& = \| f_\alpha \|_{L^2(\R \times \R^n)}^2.
\end{aligned}
\end{equation}
The inequality \eqref{in:usingT} and the identity \eqref{id:back} yield 
\eqref{in:goal} after taking square root on each side. This ends the proof of 
this lemma.
\end{proof}

We end this section by proving the inequality for the 
multiplier $T$ that we have used in the proof of \cref{lem:boundedness_S}.

\begin{lemma} \label{lem:multiplierT}\sl
Let $q$ be the polynomial function $q(\xi) = \xi_1 - \xi_2^2 + i \xi_2$. For 
every $f \in \mathcal{S}(\R^2)$, define $Tf$ as 
\[ Tf (x) = \frac{1}{2\pi} \int_{\R^2} e^{i x \cdot \xi} \frac{1}{q(\xi)} \widehat{f}(\xi) \, \dd \xi,  \qquad \forall x\in \R^2. \]
Then, there exists an absolute constant $C > 0$ such that, for every 
$\theta \in (0, 1/2)$ and every measurable subsets $E$ and $F$ of $\R^2$,
we have that
\[ \| Tf \|_{L^2(E)} \leq C |E|^\theta |F|^{1/2 - \theta} \| f \|_{L^2(\R^2)} \]
for all $f \in \mathcal{S}(\R^2) $ such that $ \supp f \subset \mathring{F} $.
\end{lemma}

\begin{proof}
Consider $\phi \in \mathcal{S}(\R^2)$ such that $ 0 \leq \phi(\xi) \leq 1 $ for 
all $\xi \in \R^2$, $\supp \phi \subset \{ \xi \in \R^2: |\xi| \leq 2 \}$ and
$ \phi(\xi) = 1 $ whenever $|\xi| \leq 1$. For $\delta > 0$, Set
\[ \phi_{\leq \delta} (\xi) = \phi ( \Re q(\xi)/ \delta, \Im q(\xi)/ \delta)  \qquad \forall \xi \in \R^2, \]
and
\[ T_{\leq \delta} f (x) = \frac{1}{2\pi} \int_{\R^2} e^{i x \cdot \xi} \frac{\phi_{\leq \delta} (\xi)}{q(\xi)} \widehat{f}(\xi) \, \dd \xi,  \qquad \forall x\in \R^2 \]
with $f \in \mathcal{S}(\R^2)$. 
If $\theta \in [1/4, 1/2)$, we choose $p = 4 \theta$ and $p^\prime$ its 
conjugate index. By H\"older's inequality first and then Hausdorff--Young 
we have that
\[ \| T_{\leq \delta} f \|_{L^\infty(\R^2)} \leq \frac{1}{2\pi} \Big( \int_{\R^2} \frac{\phi(\xi/\delta)^p}{|\xi|^p} \, \dd \xi \Big)^{1/p} \| \widehat{f} \|_{L^{p^\prime} (\R^2)} \leq C \delta^{2/p -1} \| f \|_{L^p(\R^2)} \]
with $C> 0$ an absolute constant that does not necessarily coincide with the one in the statement. Thus,
\begin{equation}
\label{in:low_freq_1}
\| T_{\leq \delta} f \|_{L^2(E)} \lesssim |E|^{1/2} |F|^{1/p - 1/2} \delta^{2/p -1} \| f \|_{L^2(\R^2)}
\end{equation}
since $ \supp f \subset \mathring{F} $.

If $\theta \in (0, 1/4]$, write $\theta^\prime = 1/2 - \theta$ so that 
$\theta^\prime \in [1/4, 1/2)$. Let $r$ and $r^\prime$ be conjugate indexes
with $r^\prime = 4 \theta^\prime$. By the Hausdorff--Young inequality
\[ \| T_{\leq \delta} f \|_{L^r(\R^2)} \leq \frac{1}{2\pi} \Big( \int_{\R^2} \frac{\phi(\xi/\delta)^{r^\prime}}{|\xi|^{r^\prime}} \, \dd \xi \Big)^{1/{r^\prime}} \| \widehat{f} \|_{L^\infty (\R^2)} \leq C \delta^{2/{r^\prime} -1} \| f \|_{L^1(\R^2)} \]
with $C> 0$ an absolute constant that does not necessarily coincide with the one in the statement. Thus,
\begin{equation}
\label{in:low_freq_2}
\| T_{\leq \delta} f \|_{L^2(E)} \lesssim |E|^{1/2 - 1/r} |F|^{1/2} \delta^{2/{r^\prime} -1} \| f \|_{L^2(\R^2)}
\end{equation}
since $ \supp f \subset \mathring{F} $.

Let $T_{> \delta}$ denote the Fourier multiplier defined by $T_{> \delta} f = T f - T_{\leq \delta} f$ for all $f \in \mathcal{S}(\R^2)$. Since
the symbol of this multiplier is the bounded function
\[ m_\delta : \xi \in \R^n \mapsto \frac{1 -\phi_{\leq \delta} (\xi)}{q(\xi)} \in \C, \]
we have that
\begin{equation}
\label{in:high_freq}
\| T_{> \delta} f \|_{L^2(\R^n)} \leq \| m_\delta \|_{L^\infty(\R^2)} \| f \| _{L^2(\R^2)} \leq \delta^{-1} \| f \| _{L^2(\R^2)}.
\end{equation}

On the one hand, using the inequalities \eqref{in:low_freq_1} and 
\eqref{in:high_freq}
we have that
\[ \| T f \|_{L^2(E)} \lesssim (|E|^{1/2} |F|^{1/p - 1/2} \delta^{2/p -1} + \delta^{-1}) \| f \|_{L^2(\R^2)}.
 \]
Choosing $\delta = |E|^{-\theta} |F|^{\theta - 1/2}$ with 
$\theta \in [1/4, 1/2)$, we have that
$$ |E|^{1/2} |F|^{1/p - 1/2} \delta^{2/p -1} = \delta^{-1},$$
and we obtain the inequality of the statement whenever $1/4 \leq \theta < 1/2$.

On the other hand, using the inequalities \eqref{in:low_freq_2} and 
\eqref{in:high_freq}
we have that
\[ \| T f \|_{L^2(E)} \lesssim (|E|^{1/2 - 1/r} |F|^{1/2} \delta^{2/{r^\prime} -1} + \delta^{-1}) \| f \|_{L^2(\R^2)}.
 \]
Choosing $\delta = |E|^{-\theta} |F|^{\theta - 1/2}$ with 
$\theta \in (0, 1/4]$, we have that
$$ |E|^{1/2 - 1/r} |F|^{1/2} \delta^{2/{r^\prime} -1} = \delta^{-1},$$
and we obtain the inequality of the statement whenever $0 < \theta \leq 1/4$.
\end{proof}

\section{The intial-to-final-state map}\label{sec:ItFmap}

\subsection{An integral identity for physical solutions}

The main goal here is to prove an integral formula that relates the 
initial-to-final-state maps with their corresponding potentials. The
proof is based on a suitable integration-by-parts identity. We end this section
by generalizing that integration-by-parts identity.

\begin{proposition}\label{prop:orthogonality} \sl Let $V_1$ and $V_2$ be 
$L^1((0,T); L^\infty(\R^n))$ with $T>0$.
For every $f, g \in L^2(\R^n)$ we have that
\begin{equation}
i \int_{\R^n} (\mathcal{U}_T^1 - \mathcal{U}_T^2) f \, \overline{g} \; = \int_\Sigma (V_1 - V_2)u_1 \overline{v_2}\,,
\label{id:integral_identity}
\end{equation}
where $u_1$ is the solution of \eqref{pb:IVP} with potential $V_1$ and initial 
data $f$, while $v_2$ is the physical solution of the following 
final-value problem 
\begin{equation}
\label{pb:final_overV2}
	\left\{
		\begin{aligned}
		& i\partial_\tm v_2 = - \Delta v_2 + \overline{V_2} v_2 & & \textnormal{in} \, \Sigma, \\
		& v_2(T, \centerdot) = g &  & \textnormal{in} \, \R^n.
		\end{aligned}
	\right.
\end{equation}
\end{proposition}

In order to prove this proposition, we need the following integration-by-parts 
formula.

\begin{proposition}\label{lem:integration_by_parts} \sl For every 
$u , v \in C([0, T]; L^2(\R^n)) $
such that $(i\partial_\tm + \Delta) u $ and $(i\partial_\tm + \Delta) v $
belong to $L^1((0,T); L^2(\R^n))$, we have
\[\int_\Sigma \big[ (i\partial_\tm + \Delta) u \overline{v} - u \overline{(i\partial_\tm + \Delta)v} \big] \, = i \int_{\R^n} \big[ 
u(T, \centerdot) \overline{v (T,\centerdot)} - u(0, \centerdot) \overline{v (0,\centerdot)} \big] \,. \]
\end{proposition}

In order to focus first on how to relate the 
initial-to-final-state maps with the corresponding potentials, we postpone the 
justification of the integration-by-parts formula after proving
the \cref{prop:orthogonality}.

\begin{proof}[Proof of the \cref{prop:orthogonality}]
Since $u_1$ and $v_2$ solve the initial-value problem \eqref{pb:IVP} 
and final-value problem \eqref{pb:final_overV2},
one has that the right-hand side of 
\eqref{id:integral_identity} is equal to
\[\int_\Sigma \big[ (i\partial_\tm + \Delta) u_1 \overline{v_2} - u_1  \overline{(i\partial_\tm + \Delta) v_2} \big] \,. \]
Since $u_1, v_2 \in C([0, T]; L^2(\R^n))$, and
the potentials $V_1$ and $\overline{V_2}$ belong to the space 
$L^1((0,T); L^\infty(\R^n))$,
the \cref{lem:integration_by_parts} yields
\begin{equation}
\int_\Sigma (V_1 - V_2)u_1 \overline{v_2} \, = i \int_{\R^n} \big[ 
u_1(T, \centerdot) \overline{v_2 (T,\centerdot)} - u_1(0, \centerdot) \overline{v_2 (0,\centerdot)} \big] \,.
\label{id:integration_by_parts}
\end{equation}
Now we introduce a third solution $u_2$ of the initial-value problem \eqref{pb:IVP} with potential $V_2$ and initial data $f$.
If, in the computations to obtain \eqref{id:integration_by_parts}, we replace $V_1$ and $u_1$ by $V_2$ and $u_2$ respectively, we can conclude that
\begin{equation}
i \int_{\R^n} \big[ 
u_2(T, \centerdot) \overline{v_2 (T,\centerdot)} - u_2(0, \centerdot) \overline{v_2 (0,\centerdot)} \big] = 0.
\label{id:trivial_orthogonality}
\end{equation}
Subtracting the left-hand side of \eqref{id:trivial_orthogonality} to the right-hand side of \eqref{id:integration_by_parts}, we get
\[\int_\Sigma (V_1 - V_2)u_1 \overline{v_2} \, = i \int_{\R^n} \big[ 
u_1(T, \centerdot) - u_2 (T, \centerdot)\big] \overline{v_2 (T,\centerdot)} \]
since $u_1(0, \centerdot) = u_2(0, \centerdot)$. Rewriting this identity in terms of the $\mathcal{U}_T^1 f$, $\mathcal{U}_T^2 f$ and $g$, we prove \eqref{id:integral_identity}.
\end{proof}

\begin{proof}[Proof of the \cref{lem:integration_by_parts}]
Start by noticing that under the assumptions of the 
\cref{lem:integration_by_parts}, we can apply the dominate convergence theorem 
to ensure that
\begin{equation}
\label{lim:delta}
\int_\Sigma \big[ (i\partial_\tm + \Delta) u \overline{v} - u \overline{(i\partial_\tm + \Delta)v} \big] \, = \lim_{\delta \to 0} \int_{\Sigma_\delta} \big[ (i\partial_\tm + \Delta) u \overline{v} - u \overline{(i\partial_\tm + \Delta)v} \big]
\end{equation}
where $\Sigma_\delta = (\delta, T - \delta) \times \R^n$.
The idea now is to approximate $u$ and $v$ by $u_\varepsilon$ and 
$v_\varepsilon$ respectively, which will be smooth in $\Sigma_\delta$ and 
compactly supported in space.

Consider a smooth cut-off $\chi \in \mathcal{S}(\R^n)$ 
such that $ 0 \leq \chi(x) \leq 1 $ for 
all $x \in \R^n$, $\supp \chi \subset \{ x \in \R^n: |x| \leq 2 \}$ and
$ \chi(x) = 1 $ whenever $|x| \leq 1$. Consider $\phi \in \mathcal{S}(\R^n)$ and
$\psi \in \mathcal{S}(\R)$ with compact supports and such that
$\supp \psi \subset \{ t \in \R : |t| \leq 1 \}$, 
$\int_{\R^n} \phi = \int_\R \psi = 1$ and $\phi(x), \psi(t) \in [0, \infty)$
for all $x \in \R^n$ and $t \in \R$. Whenever $\varepsilon > 0$, define 
$\chi_\varepsilon (x) = \chi (\varepsilon x)$ for $x \in \R^n$,
$\phi_\varepsilon (x) = \varepsilon^{-n} \phi (x/\varepsilon) $ 
for $x \in \R^n$, and 
$\psi_\varepsilon (t) = \varepsilon^{-1} \psi (t/\varepsilon) $ 
for $t \in \R$.

For $w \in C([0, T]; L^2(\R^n))$ and $\varepsilon < \delta$, 
we consider
\[w_\varepsilon: (t, x) \in \Sigma_\delta \mapsto \chi_\varepsilon (x) \int_{(0, T)} \psi_\varepsilon (t - s) \Big( \int_{\R^n} \phi_{\varepsilon} (x - y) w (s, y) \, \dd y \Big) \, \dd s.\]
One can check that
\begin{equation}
\label{lim:t-L2_approx}
\lim_{\varepsilon \to 0} \| w(t, \centerdot) - w_\varepsilon (t, \centerdot)  \|_{L^2(\R^n)} = 0  \qquad \forall t \in [\delta, T - \delta].
\end{equation}
Furthermore, if
$(i\partial_\tm + \Delta) w \in L^1((0, T); L^2(\R^n))$, then
\begin{equation}
\label{lim:L1L2_approx}
\lim_{\varepsilon \to 0} \| (i\partial_\tm + \Delta)w - (i\partial_\tm + \Delta)w_\varepsilon  \|_{L^1((\delta, T - \delta); L^2(\R^n))} = 0.
\end{equation}
Let us give an explanation about \eqref{lim:L1L2_approx}.
For notational convenience, we write 
$\varphi_\varepsilon (t, x) = \psi_\varepsilon (t) \phi_\varepsilon (x) $ 
for $(t, x) \in \R \times \R^n$, and let $\tilde{w}$ denote the 
trivial extension of $w$
\[ \tilde w (t, \centerdot) = \left\{  
\begin{aligned} 
&w(t, \centerdot) &  & t \in [0, T], \\
&0 & & t \notin [0, T].
\end{aligned}
\right.\]
Thus, we have that
$ w_\varepsilon (t, x) = \chi_\varepsilon (x) (\varphi_\varepsilon \ast \tilde{w}) (t, x) $
for all $(t,x) \in \Sigma_\delta$.
In order to show that \eqref{lim:L1L2_approx} holds, 
let us compute
\[(i\partial_\tm + \Delta)w_\varepsilon = \chi_\varepsilon (i\partial_\tm + \Delta) (\varphi_\varepsilon \ast \tilde{w}) + 2 \nabla \chi_\varepsilon \cdot \nabla (\varphi_\varepsilon \ast \tilde{w}) + \Delta \chi_\varepsilon (\varphi_\varepsilon \ast \tilde{w}) \enspace \textnormal{in} \enspace \Sigma_\delta. \]
Thus, \eqref{lim:L1L2_approx} follows from
\begin{equation}
\label{lim:chi_varphi}
\lim_{\varepsilon \to 0} \| (i\partial_\tm + \Delta)w - \chi_\varepsilon (i\partial_\tm + \Delta) (\varphi_\varepsilon \ast \tilde{w})  \|_{L^1((\delta,T - \delta); L^2(\R^n))} = 0
\end{equation}
and
\[ \lim_{\varepsilon \to 0} \Big( \| \nabla \chi_\varepsilon \cdot \nabla (\varphi_\varepsilon \ast \tilde{w}) \|_{L^1((\delta,T - \delta); L^2(\R^n))} + \| \Delta \chi_\varepsilon (\varphi_\varepsilon \ast \tilde{w}) \|_{L^1((\delta,T - \delta); L^2(\R^n))} \Big) = 0. \]
We just make a comment on how to prove \eqref{lim:chi_varphi}. 
By the triangle inequality
\begin{equation}
\label{in:triangle&others}
\begin{aligned}
\| (i\partial_\tm + \Delta)w - \chi_\varepsilon (i\partial_\tm & + \Delta) (\varphi_\varepsilon \ast \tilde{w})  \|_{L^1((\delta,T - \delta); L^2(\R^n))} \\ 
\leq & \, \| (i\partial_\tm + \Delta) w - \varphi_\varepsilon \ast [(i\partial_\tm + \Delta) \tilde{w}]  \|_{L^1((\delta,T - \delta); L^2(\R^n))} \\ 
& + \| (1 - \chi_\varepsilon)(i\partial_\tm + \Delta)w \|_{L^1((\delta,T - \delta); L^2(\R^n))}.
\end{aligned}
\end{equation}
To write the first term on the right-hand side as above, we have used
$\| \chi_\varepsilon \|_{L^\infty(\R^n)} = 1$, and 
$(i\partial_\tm + \Delta) (\varphi_\varepsilon \ast \tilde{w}) = \varphi_\varepsilon \ast [(i\partial_\tm + \Delta) \tilde{w}]$.
It is obvious that
\[ \lim_{\varepsilon \to 0} \| (1 - \chi_\varepsilon) (i\partial_\tm + \Delta)w(t, \centerdot) \|_{L^2(\R^n)} = 0  \enspace \textnormal{for almost every}\, t \in (\delta, T - \delta). \]
Hence, by the dominate convergence theorem,
\[ \lim_{\varepsilon \to 0} \| (1 - \chi_\varepsilon)(i\partial_\tm + \Delta)w \|_{L^1((\delta, T - \delta); L^2(\R^n))} = 0. \]
To conclude that \eqref{lim:chi_varphi} holds, we still have to analyse the 
first term on the right-hand side of the inequality \eqref{in:triangle&others}.
Note that since $\varepsilon < \delta$, we have that
\[ \varphi_\varepsilon \ast [(i\partial_\tm + \Delta) \tilde{w}] (t, x) = \int_\Sigma \varphi_\varepsilon (t - s, x- y) [(i\partial_\tm + \Delta)w] (s, y) \, \dd (s, y) \quad \textnormal{a.e.} \, (t, x) \in \Sigma_\delta. \]
Hence, one can check by using standard arguments that
\[\lim_{\varepsilon \to 0} \| (i\partial_\tm + \Delta) w - \varphi_\varepsilon \ast [(i\partial_\tm + \Delta) \tilde{w}]  \|_{L^1((\delta,T - \delta); L^2(\R^n))} = 0.\]
Therefore, \eqref{lim:L1L2_approx} holds.

We now apply these results for the corresponding 
$u_\varepsilon$ and $v_\varepsilon$.
By the standard integration-by-parts rules we have that
\[\int_{\Sigma_\delta} \big[ (i\partial_\tm + \Delta) u_\varepsilon \overline{v_\varepsilon} - u_\varepsilon \overline{(i\partial_\tm + \Delta)v_\varepsilon} \big] \, = i \int_{\R^n} \big[ 
u_\varepsilon(T - \delta, \centerdot) \overline{v_\varepsilon (T - \delta,\centerdot)} - u_\varepsilon(\delta, \centerdot) \overline{v_\varepsilon (\delta,\centerdot)} \big] \,. \]
After \eqref{lim:t-L2_approx} and \eqref{lim:L1L2_approx} we can conclude,
by the dominate convergence theorem, that
\[\int_{\Sigma_\delta} \big[ (i\partial_\tm + \Delta) u \overline{v} - u \overline{(i\partial_\tm + \Delta)v} \big] \, = i \int_{\R^n} \big[ 
u(T - \delta, \centerdot) \overline{v (T - \delta,\centerdot)} - u(\delta, \centerdot) \overline{v (\delta,\centerdot)} \big] \,. \]
The limits \eqref{lim:delta} and
\[ \int_{\R^n} \big[ 
u(T, \centerdot) \overline{v (T,\centerdot)} - u(0, \centerdot) \overline{v (0,\centerdot)} \big] = \lim_{\delta \to 0} \int_{\R^n} \big[ 
u(T - \delta, \centerdot) \overline{v (T - \delta,\centerdot)} - u(\delta, \centerdot) \overline{v (\delta,\centerdot)} \big]\]
yield the identity in the statement.
\end{proof}

We end this section by proving a generalization of the integration-by-parts
formula that will be useful at a later stage.
\begin{proposition}\label{lem:_generalization_integration_by_parts} \sl 
For every $u\in C([0, T]; L^2(\R^n)) $ such that 
$(i\partial_\tm + \Delta) u \in L^2(\Sigma) $ and
$u(0,\centerdot) = u(T, \centerdot) = 0$,
and every $v \in L^2(\tilde \Sigma)$ such that
$(i\partial_\tm + \Delta) v \in L^2(\tilde \Sigma)$ with
$\tilde{\Sigma} = (-\delta, T + \delta) \times \R^n$ and $\delta > 0$, we have
\[\int_\Sigma (i\partial_\tm + \Delta) u \overline{v} = \int_\Sigma u \overline{(i\partial_\tm + \Delta)v} \,. \]
\end{proposition}

\begin{proof}
Consider $\tilde{u} \in C([-\delta, T + \delta]; L^2(\R^n))$ 
the trivial extension of $u$ given by
\[\tilde u (t, x) = \left\{ \begin{aligned}
& u (t, x) & & (t, x) \in \Sigma,\\
& 0 & & (t, x) \in \tilde \Sigma \setminus \Sigma.
\end{aligned} \right.\]
Note that
\[ \langle (i\partial_\tm + \Delta) \tilde u, \phi \rangle = \int_\Sigma (i\partial_\tm + \Delta) u \, \phi \]
for all $\phi \in \mathcal{D}(\tilde{\Sigma})$, which implies that
$(i\partial_\tm + \Delta) \tilde{u} \in L^2(\tilde \Sigma)$.

For  $\varepsilon \in (0, \delta)$, define
\[ u_\varepsilon : (t, x) \in \tilde{\Sigma} \mapsto \chi_\varepsilon(x) [\varphi_\varepsilon \ast \tilde{u}] (t, x) \]
with $\varphi_\varepsilon (t, x) = \psi_\varepsilon (t) \phi_\varepsilon (x)$
for all $(t, x) \in \R \times \R^n$, and $\chi_\varepsilon$, $\phi_\varepsilon$ 
and $\psi_\varepsilon$ as in the proof of the \cref{lem:integration_by_parts}.
Then,
\[ \int_\Sigma u \overline{(i\partial_\tm + \Delta)v} \, = \int_{\tilde{\Sigma}} (\tilde{u} - u_\varepsilon) \overline{(i\partial_\tm + \Delta)v} \,
+ \int_{\tilde{\Sigma}} u_\varepsilon \overline{(i\partial_\tm + \Delta)v} \]
with
\[ \lim_{\varepsilon \to 0} \| \tilde{u} - u_\varepsilon \|_{L^2(\tilde \Sigma)} = 0. \]
This implies that
\begin{equation}
\label{lim:left}
\int_\Sigma u \overline{(i\partial_\tm + \Delta)v} \, = \lim_{\varepsilon \to 0} \int_{\tilde{\Sigma}} u_\varepsilon \overline{(i\partial_\tm + \Delta)v}.
\end{equation}
Since $u_\varepsilon \in \mathcal{D} (\tilde{\Sigma})$, we have by the 
distributional definition of $i\partial_\tm + \Delta$ that
\begin{equation}
\label{id:distributional_derivative}
\int_{\tilde{\Sigma}} u_\varepsilon \overline{(i\partial_\tm + \Delta)v} \, = \int_{\tilde{\Sigma}} (i\partial_\tm + \Delta)u_\varepsilon \overline{v} \, .
\end{equation}
Additionally,
\begin{align*}
(i\partial_\tm + \Delta)u_\varepsilon = \chi_\varepsilon (i\partial_\tm + \Delta) [\varphi_\varepsilon \ast \tilde{u}] + 2 \nabla \chi_\varepsilon \cdot \nabla (\varphi_\varepsilon \ast \tilde{u}) + \Delta \chi_\varepsilon (\varphi_\varepsilon \ast \tilde{u}).
\end{align*}
Since
$(i\partial_\tm + \Delta) [\varphi_\varepsilon \ast \tilde{u}] = \varphi_\varepsilon \ast (i\partial_\tm + \Delta) \tilde{u} $, and
$(i\partial_\tm + \Delta) \tilde{u} \in L^2(\tilde \Sigma)$, we have that
\[ \lim_{\varepsilon \to 0} \| (i\partial_\tm + \Delta) \tilde{u} - \chi_\varepsilon (i\partial_\tm + \Delta) [\varphi_\varepsilon \ast \tilde{u}] \|_{L^2(\tilde \Sigma)} = 0.\]
One can check that
\begin{equation}
\label{lim:left-over}
\lim_{\varepsilon \to 0} \Big( \| 2 \nabla \chi_\varepsilon \cdot \nabla (\varphi_\varepsilon \ast \tilde{u}) \|_{L^2(\tilde \Sigma)} + \| \Delta \chi_\varepsilon (\varphi_\varepsilon \ast \tilde{u}) \|_{L^2(\tilde \Sigma)} \Big) = 0.
\end{equation}
Since $v \in L^2(\tilde \Sigma)$, we have that
\begin{equation}
\label{lim:right}
\lim_{\varepsilon \to 0} \int_{\tilde{\Sigma}} (i\partial_\tm + \Delta)u_\varepsilon \overline{v} \, = \int_{\Sigma} (i\partial_\tm + \Delta) u \overline{v} \, .
\end{equation}
Taking limits as $\varepsilon$ goes to zero at each side of the identity 
\eqref{id:distributional_derivative}, and using \eqref{lim:left} and 
\eqref{lim:right}, we obtain the identity stated.
\end{proof}

\subsection{A collection of technical results}
In this section, we will present and prove a number of results that will be 
useful in the next section to obtain the orthogonal relation for
pairs of CGO solutions.

It is convenient to recall the
well-posedness of the following problem.
\begin{equation}
\label{pb:non-homogenous}
	\left\{
		\begin{aligned}
		& (i\partial_\tm + \Delta - V) u = F & & \textnormal{in} \, \Sigma, \\
		& u(0, \centerdot) = 0 &  & \textnormal{in} \, \R^n.
		\end{aligned}
	\right.
\end{equation}
Consider $V \in L^1((0,T); L^\infty(\R^n))$. Then, for every 
$ F \in L^1((0,T); L^2(\R^n)) $, there exists a unique 
$ u \in C([0, T]; L^2(\R^n)) $
solving the problem \eqref{pb:non-homogenous}. Moreover, the linear map 
$ F \in L^1((0,T); L^2(\R^n)) \mapsto u \in C([0, T]; L^2(\R^n)) $ is bounded.
These facts have been proved for example in \cite{zbMATH05013664}.

\begin{lemma} \label{lem:orthogonal_rewardPB} \sl Consider $V \in L^1((0,T); L^\infty(\R^n))$ and 
$ F \in L^1((0,T); L^2(\R^n)) $ such that
\[ \int_\Sigma F \overline{v} \, = 0 \]
for $v \in C([0, T]; L^2(\R^n))$ solution of the final-value problem
\[
\left\{
		\begin{aligned}
		& i\partial_\tm v = - \Delta v + \overline{V} v & & \textnormal{in} \, \Sigma, \\
		& v(T, \centerdot) = g &  & \textnormal{in} \, \R^n.
		\end{aligned}
	\right.
\]
with every $g \in L^2 (\R^n)$. Then, the solution $u \in C([0, T]; L^2(\R^n))$
of the problem \eqref{pb:non-homogenous} satisfies that
\[ u(T, \centerdot) = 0. \]
\end{lemma}

\begin{proof}
Consider $g \in L^2(\R^n)$, and let $v$ denote the solution
of the corresponding final-value problem. Then, by the 
\cref{lem:integration_by_parts} we have
\[ i \int_{\R^n} u(T, \centerdot) \overline{g} \, = \int_\Sigma \big[ (i\partial_\tm + \Delta) u \overline{v} - u \overline{(i\partial_\tm + \Delta)v} \big] \, .  \]
Note that the \cref{lem:integration_by_parts} can be applied because 
$ F \in L^1((0,T); L^2(\R^n)) $, 
$V \in L^1((0,T);L^\infty(\R^n))$, and $u$ and $v$ belongs to 
$C([0, T]; L^2(\R^n))$.
Adding and subtracting $V u \overline{v}$ we have, by 
\eqref{pb:non-homogenous}, that
\[ i \int_{\R^n} u(T, \centerdot) \overline{g} \, = \int_\Sigma F \overline{v} \, - \int_\Sigma u \overline{(i\partial_\tm + \Delta - \overline{V})v} \, .  \]
The first term on the right-hand side is assumed to vanish, 
while the second term vanishes because $v$ is
the solution of the final-value problem.
Hence,
\[\int_{\R^n} u(T, \centerdot) \overline{g} \, = 0\]
Since $g$ is arbitrary, we can conclude that $u(T, \centerdot) = 0$.
\end{proof}

There is a symmetric version of this lemma, that we state without proof because
of its symmetry with respect to the previous one.
\begin{lemma}\label{lem:orthogonal_forwardPB} \sl Consider $V \in L^1((0,T); L^\infty(\R^n))$ and 
$ G \in L^1((0,T); L^2(\R^n)) $ such that
\[ \int_\Sigma \overline{G} u \, = 0 \]
for $u \in C([0, T]; L^2(\R^n))$ solution of the initial-value problem 
\eqref{pb:IVP} with every $f \in L^2 (\R^n)$.
Then, the solution $v \in C([0, T]; L^2(\R^n))$
of the problem 
\begin{equation}
\label{pb:zero-final_non-homo}
\left\{
		\begin{aligned}
		& (i\partial_\tm + \Delta - \overline{V}) v = G & & \textnormal{in} \, \Sigma, \\
		& v(T, \centerdot) = 0 &  & \textnormal{in} \, \R^n.
		\end{aligned}
	\right.
\end{equation}
satisfies that
\[ v(0, \centerdot) = 0. \]
\end{lemma}

For the next result it is convenient to recall that, for
$\nu \in \R^n \setminus \{ 0 \}$, we introduced the notation 
\[ \Sigma^\nu_j = (0, T) \times \Omega^\nu_j \]
right after the \eqref{cond:V_CGO} in introduction.

\begin{lemma}\label{lem:exp_decay} \sl Consider $F \in L^2(\Sigma)$ such that 
there exists $c > 0$ so that $e^{c|\x|} F \in L^2(\Sigma)$.
Let $u \in C([0,T]; L^2(\R^n))$ satisfy the conditions
\[
\left\{
		\begin{aligned}
		& (i\partial_\tm + \Delta) u = F & & \textnormal{in} \enspace \Sigma, \\
		& u(0, \centerdot) = u(T, \centerdot) = 0 &  & \textnormal{in} \enspace \R^n.
		\end{aligned}
	\right.
\]
Then, there exists an absolute constant $C > 0$ such that
\[ \| e^{\nu \cdot \x} u \|_{L^2(\Sigma^\nu_{< R})} \leq C |\nu|^{-1/2} T^{1/2} R^\theta \sum_{j \in \N_0 } 2^{j(1/2 - \theta)} \| e^{\nu \cdot \x} F \|_{L^2(\Sigma_j^\nu)}\]
for all $\nu \in \R^n \setminus \{ 0 \}$ with $|\nu| < c$ and 
all $\theta \in (0, 1/2)$. Here 
$\Sigma^\nu_{< R} = (0, T) \times \{ x \in \R^n : |x \cdot \nu| < |\nu| R \}$.
\end{lemma}

In order to prove this lemma, we will need to state some well known facts 
(see for example Theorem 7.4.2 in \cite{zbMATH01950198}). For 
convenience for the readers, we include a proof of these facts.

\begin{lemma}\label{lem:Fourier_hol-extension}\sl Consider $F \in L^1(\R; L^2(\R^n))$ such that
there exists $c > 0$ so that $e^{c|\x|} F \in L^1(\R; L^2(\R^n))$. Then:
\begin{enumerate}[label=\textnormal{(\alph*)}, ref=\textnormal{\alph*}]
\item \label{ite:L1} For every $\nu \in \R^n$ such that 
$|\nu| < c$, we have $e^{\nu \cdot \x} F \in L^1(\R \times \R^n)$.
\item \label{ite:representation} The Fourier transform of $F$ can be extended to 
$\R \times \R^{n - 1} \times \C_{|\Im|<c}$  as the continuous function
\[ \widehat{F} : (\tau, \xi^\prime, \zeta) \mapsto \frac{1}{(2 \pi)^\frac{n+1}{2}} \int_{\R \times \R^n } e^{-i(\tau t + (\xi^\prime, \zeta) \cdot x)} F(t, x) \, \dd (t, x), \]
where $\C_{|\Im|<c} = \{ \zeta \in \C : |\Im \zeta| < c \}$. Furthermore,
it satisfies the relation 
\[ \widehat{e^{\lambda \x_n} F}(\tau, \xi) = \widehat{F}(\tau, \xi^\prime, \xi_n + i \lambda) \]
for all $(\tau, \xi) \in \R \times \R^n$ and $\lambda \in \R$ such that 
$|\lambda|<c$. Here $\x_n = e_n \cdot \x$ with $e_n$ the $n^{\rm th}$ element of 
the standard basis of $\R^n$, and $\xi = (\xi^\prime, \xi_n)$.
\item \label{ite:holomrphic_z} For every 
$(\tau, \xi^\prime) \in \R \times \R^{n-1}$, the function 
\[ \zeta \in \C_{|\Im|<c} \mapsto \widehat{F}(\tau, \xi^\prime, \zeta) \]
is holomorphic in $\C_{|\Im|<c} $.
\end{enumerate}
\end{lemma}

\begin{proof}
Start by proving \eqref{ite:L1}. Note that
\[\nu \cdot x \leq - ( c - |\nu|) |x| + c |x|\]
for all $x \in \R^n$. Hence
\[ \|e^{\nu \cdot \x} F \|_{L^1 (\R \times \R^n)} \leq \| e^{- ( c - |\nu|) |x|} \|_ {L^2(\R^n)} \| e^{c |\x|} F \|_{L^1(\R; L^2(\R^n))}, \]
where $e^{- ( c - |\nu|) |x|} \in L^2(\R^n) $ since $|\nu| < c$.

The property \eqref{ite:representation} 
is a direct consequence of \eqref{ite:L1} since
for $(\tau, \xi^\prime, \zeta) \in \R \times \R^{n - 1} \times \C_{|\Im|<c}$,
we have that
\begin{align*}
\widehat{e^{\Im \zeta e_n \cdot \x} F} (\tau, \xi^\prime, \Re \zeta) & = \frac{1}{(2 \pi)^\frac{n+1}{2}} \int_{\R \times \R^n } e^{-i(\tau t + (\xi^\prime, \Re \zeta) \cdot x)} e^{\Im \zeta e_n \cdot x} F(t, x) \, \dd (t, x) \\
& = \frac{1}{(2 \pi)^\frac{n+1}{2}} \int_{\R \times \R^n } e^{-i(\tau t + (\xi^\prime, \zeta) \cdot x)} F(t, x) \, \dd (t, x).
\end{align*}
Hence, we can use the last expression to extend $\widehat{F}$ to
$\R \times \R^{n - 1} \times \C_{|\Im|<c}$. Furthermore, choosing 
$\zeta = \xi_n + i \lambda$ we have that
\[ \widehat{e^{\lambda \x_n} F} (\tau, \xi^\prime, \xi_n) = \widehat{F} (\tau, \xi^\prime, \xi_n + i \lambda).  \]
This proves the property \eqref{ite:representation}.

Finally, to verify the property \eqref{ite:holomrphic_z} one 
needs to check that the function satisfies the Cauchy--Riemann equations. This 
can be justified since 
$\x_n e^{\Im \zeta e_n \cdot \x} F \in L^1(\R \times \R^n)$. 
This completes the proof of \eqref{ite:holomrphic_z}.
\end{proof}

Before proving the \cref{lem:exp_decay}, we still need a preparatory results, 
which is a consequence of the \cref{lem:Fourier_hol-extension}.

\begin{corollary}\label{cor:vanish}\sl 
Consider $F \in L^1((0,T); L^2(\R^n))$ such that
there exists $c > 0$ so that $e^{c|\x|} F \in L^1((0,T); L^2(\R^n))$.
Let $G \in L^1(\R; L^2(\R^n))$ denote the function
\begin{equation*}
G (t, \centerdot) =
	\left\{
		\begin{aligned}
		& F(t, \centerdot) & & \textnormal{if} \, t \in (0,T), \\
		& 0 &  & \textnormal{if} \, t \in \R \setminus (0,T).
		\end{aligned}
	\right.
\end{equation*}
If there exists $u \in C([0,T]; L^2(\R^n))$ satisfying
\[
\left\{
		\begin{aligned}
		& (i\partial_\tm + \Delta) u = F & & \textnormal{in} \enspace \Sigma, \\
		& u(0, \centerdot) = u(T, \centerdot) = 0 &  & \textnormal{in} \enspace \R^n;
		\end{aligned}
	\right.
\]
then, for every $(\tau, \xi^\prime) \in \R \times \R^{n-1}$ such that 
$\tau \neq -|\xi^\prime|^2$, the function 
\begin{equation}
\label{map:Gzparoboloid}
\zeta \in \C_{|\Im|<c} \mapsto \frac{\widehat{G}(\tau, \xi^\prime, \zeta)}{-\tau - |\xi^\prime|^2 - \zeta^2}
\end{equation}
is holomorphic in $\C_{|\Im|<c} = \{ \zeta \in \C : |\Im \zeta| < c \}$.
\end{corollary}

\begin{proof} From the item \eqref{ite:holomrphic_z} in the 
\cref{lem:Fourier_hol-extension} we know that, for every 
$(\tau, \xi^\prime) \in \R \times \R^{n-1}$, the function 
\begin{equation}
\label{map:Gz}
\zeta \in \C_{|\Im|<c} \mapsto \widehat{G}(\tau, \xi^\prime, \zeta)
\end{equation}
is holomorphic in $\C_{|\Im|<c} $. While, whenever $\tau \neq - |\xi^\prime|^2$, 
the function 
\[ \zeta \in \C_{|\Im|<c} \mapsto \frac{1}{-\tau - |\xi^\prime|^2 - \zeta^2} \]
is meromorphic, with simple poles at $\zeta_1 = -|\tau + |\xi^\prime|^2 |^{1/2} $ 
and $ \zeta_2 = |\tau + |\xi^\prime|^2 |^{1/2} $ if $- \tau - |\xi^\prime|^2 > 0$, and $\zeta_3 = -i|\tau + |\xi^\prime|^2 |^{1/2} $ and 
$ \zeta_4 = i|\tau + |\xi^\prime|^2 |^{1/2} $ if $- \tau - |\xi^\prime|^2 < 0$
and $|\tau + |\xi^\prime|^2 | < c^2$. 
Then, in order to prove that the function 
\eqref{map:Gzparoboloid} is holomorphic is enough to check that \eqref{map:Gz} vanishes at these poles. Let us focus on proving that.

The item \eqref{ite:L1} in the \cref{lem:Fourier_hol-extension} and the fact that
$\supp G \subset \overline{\Sigma}$ ensure that, for every 
$(\sigma, \nu) \in \R \times \R^n$ such that  $|\nu| < c$, 
we have $e^{\sigma \tm + \nu \cdot \x} G \in L^1(\R \times \R^n)$.
Thus, the Fourier transform of $G$ can be extended to 
$\C \times \R^{n - 1} \times \C_{|\Im|<c}$  as the continuous function
\begin{equation}
\label{map:Ghatext}
\widehat{G} : (\gamma, \xi^\prime, \zeta) \mapsto \frac{1}{(2 \pi)^\frac{n+1}{2}} \int_{\R \times \R^n } e^{-i(\gamma t + (\xi^\prime, \zeta) \cdot x)} G(t, x) \, \dd (t, x).
\end{equation}
One can check that, for every $\xi^\prime \in \R^{n-1}$, the function 
\begin{equation}
\label{map:Gzz}
\zeta \in \C_{|\Im|<c} \mapsto \widehat{G}(-|\xi^\prime|^2 - \zeta^2, \xi^\prime, \zeta)
\end{equation}
is holomorphic in $\C_{|\Im|<c} $. We will see later that
$\widehat{G}(-|\xi|^2, \xi) = 0$ for all $\xi \in \R^n$, that is,
for every $\xi^\prime \in \R^{n-1}$ the function
\eqref{map:Gzz} vanishes on $\{ \zeta \in \C_{|\Im|<c} : \Im \zeta = 0 \}$.
Then, by analytic 
continuation we have that
\[ \widehat{G} (-|\xi^\prime|^2 - \zeta^2, \xi^\prime, \zeta) = 0  \qquad \forall (\xi^\prime, \zeta) \in \R^{n - 1} \times \C_{|\Im| < c},\]
which implies that \eqref{map:Gz} vanishes at the poles 
$\{ \zeta_1, \zeta_2, \zeta_3, \zeta_4\}$.
Indeed, if $- \tau - |\xi^\prime|^2 > 0$ we have that
\[ \widehat{G} (\tau, \xi^\prime, \zeta_j) =  \widehat{G} (-|\xi^\prime|^2 - \zeta_j^2, \xi^\prime, \zeta_j) = 0 \qquad j \in \{ 1, 2 \}; \]
however, $- \tau - |\xi^\prime|^2 < 0$ and $|\tau + |\xi^\prime|^2 | < c^2$
we have that
\[ \widehat{G} (\tau, \xi^\prime, \zeta_j) =  \widehat{G} (-|\xi^\prime|^2 - \zeta_j^2, \xi^\prime, \zeta_j) = 0 \qquad j \in \{ 3, 4 \}. \]
Therefore, the function \eqref{map:Gz} vanishes at the poles.

To conclude the proof of this corollary, we need to show that
$\widehat{G}(-|\xi|^2, \xi) = 0$ for all $\xi \in \R^n$. By the expression
\eqref{map:Ghatext}, we have that
\[ \widehat{G}(-|\xi|^2, \xi) = \frac{1}{(2 \pi)^\frac{n+1}{2}} \int_\Sigma e^{-i(-|\xi|^2 t + \xi \cdot x)} F(t, x) \, \dd (t, x). \]
Since every arbitrary function in $ \mathcal{S}(\R^n)$ can be written as 
$\widehat{\phi}$ with $\phi \in \mathcal{S}(\R^n)$, we have
\[ \int_{\R^n} \widehat{G}(-|\xi|^2, \xi) \overline{\widehat{\phi(\xi)}} \, \dd \xi = \frac{1}{(2 \pi)^\frac{1}{2}} \int_\Sigma (i\partial_\tm + \Delta) u \, \overline{[e^{i\tm \Delta} \phi]} \]
using the fact that $(i\partial_\tm + \Delta) u = F$ in $\Sigma$. The 
\cref{lem:integration_by_parts} and the fact that 
$u(0, \centerdot) = u(T, \centerdot) = 0$ imply that the right-hand side of
the previous identity vanishes. Hence $\widehat{G}(-|\xi|^2, \xi) = 0$ for all 
$\xi \in \R^n$, which concludes this proof.
\end{proof}

\begin{proof}[Proof of the \cref{lem:exp_decay}]
Start by checking that the right-hand side of the 
inequality in the statement of this lemma is finite for all $\theta \in (0,1/2)$. 
Indeed, the chain of inequalities
\begin{equation}
\label{in:exponents}
\nu \cdot x \leq - \Big( \frac{c}{|\nu|} - 1 \Big) |\nu \cdot x| +  \frac{c}{|\nu|} |\nu \cdot x| \leq - \Big( \frac{c}{|\nu|} - 1 \Big) |\nu \cdot x| +  c |x|,
\end{equation}
valid for all $x \in \R^n$, implies that
\begin{equation}
\label{in:exp_decay}
\| e^{\nu \cdot \x} F \|_{L^2(\Sigma_j^\nu)} \leq \| e^{-( \frac{c}{|\nu|} - 1 )|\nu \cdot \x|} \|_{L^\infty(\Sigma_j^\nu)} \| e^{c|\x|} F \|_{L^2(\Sigma)}.
\end{equation}
Consequently,
\[ \sum_{j \in \N_0 } 2^{j(1/2 - \theta)} \| e^{\nu \cdot \x} F \|_{L^2(\Sigma_j^\nu)} \leq \Big( 1 +  \sum_{j \in \N } 2^{j(1/2 - \theta)} e^{-( \frac{c}{|\nu|} - 1 )2^{j - 1}} \Big) \| e^{c|\x|} F \|_{L^2(\Sigma)}, \]
where the series between parenthesis is convergent. This shows that the right-hand 
side of the inequality stated in the \cref{lem:exp_decay} is finite for all 
$\theta \in (0,1/2)$.

For $\nu \in \R^n \setminus \{ 0 \}$ and $g \in \mathcal{S}(\R \times \R^n)$,
recall the definition of the multiplier $S_{(0,-\nu)}$, that will be denoted here
by $T_\nu $:
\[T_\nu g (t, x) = \frac{1}{(2 \pi)^\frac{n+1}{2}} \int_{\R \times \R^n} e^{i(t \tau + x \cdot \xi)} \frac{\widehat{g} (\tau, \xi)}{-\tau - (\xi + i \nu)\cdot(\xi + i \nu)} \, \dd (\tau, \xi) \]
for all $(t, x) \in \R \times \R^n$. Changing variables according to 
$ 4 |\nu|^2 \sigma = - \tau + |\nu|^2$, $2 |\nu| Q \eta = \xi$ with $Q \in \Orth(n)$ such that $- \nu = |\nu| Q e_n$, and 
$(2 |\nu|)^{n + 2} \dd (\sigma, \eta) = \dd (\tau, \xi)$:
\begin{align*}
& T_\nu g (t, x) = \frac{e^{i|\nu|^2 t}}{4|\nu|^2} S f_\nu (- 4 |\nu|^2 t, 2 |\nu| Q^\T x),\\
& f_\nu (s, y) = e^{is/4} g\Big(-\frac{s}{4|\nu|^2}, \frac{Q y}{2|\nu|} \Big);
\end{align*}
where the multiplier $S$ is defined in \eqref{def:S}.
Using the previous relations and the \cref{lem:boundedness_S} 
with $I = (- 4 |\nu|^2 T, 0)$, $J = (-2|\nu|R, 2|\nu|R)$, 
$I_j = - 4 |\nu|^2 \Upsilon_j$ and $J_k = 2 |\nu| \Upsilon_k$,
one can check that
there exists an absolute constant $C > 0$ such that
\begin{equation}
\label{in:boundedness_Tnu}
\| T_\nu g \|_{L^2(\Sigma^\nu_{< R})} \leq C |\nu|^{-1/2} T^\theta R^\theta \sum_{\alpha \in \N_0^2} 2^{|\alpha|(1/2 - \theta)} \| g \|_{L^2(\Pi^\nu_\alpha)}
\end{equation}
for all $\nu \in \R^n \setminus \{ 0 \}$, $g \in \mathcal{S}(\R \times \R^n)$ 
and $\theta \in (0, 1/2)$.
Recall that the sequence of sets 
$\{ \Pi_\alpha^\nu : \alpha \in \N_0^2 \}$ was described in the 
\cref{def:Xnu}.

Let $G \in L^2(\R \times \R^n)$ denote the function
\begin{equation*}
G (t, \centerdot) =
	\left\{
		\begin{aligned}
		& F(t, \centerdot) & & \textnormal{if} \, t \in (0,T), \\
		& 0 &  & \textnormal{if} \, t \in \R \setminus (0,T).
		\end{aligned}
	\right.
\end{equation*}
As a consequence of the inequality \eqref{in:boundedness_Tnu}, $T_\nu$ can be 
extended to define $T_\nu (e^{\nu \cdot \x} G)$ for every 
$\nu \in \R^n \setminus \{ 0 \}$ such that $|\nu| < c$ since
\[\sum_{\alpha \in \N_0^2} 2^{|\alpha|(1/2 - \theta)} \|e^{\nu \cdot \x} G \|_{L^2(\Pi_\alpha^\nu)} \leq C T^{1/2 - \theta} \sum_{j \in \N_0} 2^{j(1/2 - \theta)} \|e^{\nu \cdot \x} F \|_{L^2(\Sigma_j^\nu)}\]
for all $\nu \in \R^n \setminus \{ 0 \}$ such that $|\nu| < c$
and all $\theta \in (0, 1/2)$. Furthermore,
\begin{equation}
\label{in:TnuExpG}
\| T_\nu (e^{\nu \cdot \x} G) \|_{L^2(\Sigma^\nu_{< R})} \leq C |\nu|^{-1/2} T^{1/2} R^\theta \sum_{j \in \N_0} 2^{j(1/2 - \theta)} \|e^{\nu \cdot \x} F \|_{L^2(\Sigma_j^\nu)}
\end{equation}
for all $\nu \in \R^n \setminus \{ 0 \}$ such that $|\nu| < c$ and all $\theta \in (0, 1/2)$.

For $\nu \in \R^n \setminus \{ 0 \}$ define
\[ \Gamma_\nu = \{ (\tau, \xi) \in \R \times \R^n : \tau = - (\xi + i \nu)\cdot(\xi + i \nu) \}. \]
Since the codimension of $\Gamma_\nu$ is $2$, we have by
the item \eqref{ite:L1} of \cref{lem:Fourier_hol-extension}
that ---for every $\phi \in \mathcal{D}(\R \times \R^n)$ and every $\nu \in \R^n \setminus \{ 0 \}$ such that $|\nu| < c$--- the function
\[ (\tau, \xi) \in (\R \times \R^n) \setminus \Gamma_\nu \mapsto \frac{\widehat{e^{\nu \cdot \x} G} (\tau, \xi)}{- \tau - (\xi + i \nu)\cdot(\xi + i \nu)} \, \widecheck{e^{-\nu \cdot \x} \phi}(\tau, \xi) \]
can be extended to $\R \times \R^n$ so that it represents an element in 
$L^1(\R \times \R^n)$. Here 
\[ \widecheck{e^{-\nu \cdot \x} \phi}(\tau, \xi) = \widehat{e^{-\nu \cdot \x} \phi}(-\tau, -\xi).  \]
Hence, we can write
\begin{align*}
\langle e^{- \nu \cdot \x}  T_\nu (e^{\nu \cdot \x} G), \phi \rangle & = \langle  T_\nu (e^{\nu \cdot \x} G), e^{-\nu \cdot \x} \phi\rangle \\
& = \int_{\R \times \R^n} \frac{\widehat{e^{\nu \cdot \x} G} (\tau, \xi)}{- \tau - (\xi + i \nu)\cdot(\xi + i \nu)} \, \widecheck{e^{-\nu \cdot \x} \phi}(\tau, \xi) \, \dd (\tau, \xi)
\end{align*}
for all $\phi \in \mathcal{D}(\R \times \R^n)$ and 
$\nu \in \R^n \setminus \{ 0 \}$ such that $|\nu| < c$. 
For every $\nu \in \R^n \setminus \{ 0 \}$, 
there exists a $Q \in \Orth(n)$ so 
that $\nu = |\nu| Q e_n$, with $e_n$ the $n^{\rm th}$ element of the standard 
basis of $\R^n$. Let $G_Q$ and $\phi_Q$ denote the functions 
$G_Q (t, x) = G (t, Qx)$ and $\phi_Q (t, x) = \phi (t, Qx)$ for all 
$(t, x) \in \R \times \R^n$. Changing variables, we have that
\[ \langle e^{- \nu \cdot \x}  T_\nu (e^{\nu \cdot \x} G), \phi \rangle = \int_{\R \times \R^n} \frac{\widehat{e^{|\nu| \x_n} G_Q} (\tau, \xi)}{- \tau - (\xi + i |\nu| e_n)\cdot(\xi + i |\nu| e_n)} \, \widecheck{e^{-|\nu| \x_n} \phi_Q}(\tau, \xi) \, \dd (\tau, \xi) \]
recall that $\x_n = e_n \cdot \x$. 
By the item \eqref{ite:representation} of the \cref{lem:Fourier_hol-extension},
we have that
the right-hand side of the previous identity is equal to
\[ \int_{\R \times \R^{n-1} \times \R} \frac{\widehat{G_Q} (\tau, \xi^\prime, \xi_n + i |\nu|)}{- \tau - |\xi^\prime|^2 - (\xi_n + i |\nu|)^2} \, \widecheck{\phi_Q}(\tau, \xi^\prime, \xi_n + i |\nu|) \, \dd (\tau, \xi^\prime, \xi_n). \]
By the Fubini--Tonelli theorem the function
\[ \xi_n \in \R \mapsto \frac{\widehat{G_Q} (\tau, \xi^\prime, \xi_n + i |\nu|)}{- \tau - |\xi^\prime|^2 - (\xi_n + i |\nu|)^2} \, \widecheck{\phi_Q}(\tau, \xi^\prime, \xi_n + i |\nu|) \]
belongs to $L^1(\R)$ for a.e. $(\tau, \xi^\prime) \in \R \times \R^{n-1}$,
the function
\[ H_Q^{|\nu|} : (\tau, \xi^\prime) \in \R \times \R^{n-1} \mapsto \int_\R \frac{\widehat{G_Q} (\tau, \xi^\prime, \xi_n + i |\nu|)}{- \tau - |\xi^\prime|^2 - (\xi_n + i |\nu|)^2} \, \widecheck{\phi_Q}(\tau, \xi^\prime, \xi_n + i |\nu|) \, \dd \xi_n \]
belongs to $L^1(\R \times \R^{n - 1})$, and
\[ \langle e^{- \nu \cdot \x}  T_\nu (e^{\nu \cdot \x} G), \phi \rangle = \int_{\R \times \R^{n - 1}} H_Q^{|\nu|}(\tau, \xi^\prime) \, \dd (\tau, \xi^\prime).  \]

Note that, by the \cref{cor:vanish} and the \cref{lem:Fourier_hol-extension}
the function
\begin{equation}
\label{map:GQxin}
\xi_n \in \R \mapsto \frac{\widehat{G_Q} (\tau, \xi^\prime, \xi_n)}{- \tau - |\xi^\prime|^2 - \xi_n^2} \, \widecheck{\phi_Q}(\tau, \xi^\prime, \xi_n)
\end{equation}
belongs to $L^1(\R)$ for all $(\tau, \xi^\prime) \in \R \times \R^{n-1}$ such that
$\tau \neq - |\xi^\prime|^2$. 

Our final goal will be to show that the function
\[H_Q : (\tau, \xi^\prime) \in \R \times \R^{n-1} \mapsto \int_\R \frac{\widehat{G_Q} (\tau, \xi^\prime, \xi_n)}{- \tau - |\xi^\prime|^2 - \xi_n^2} \, \widecheck{\phi_Q}(\tau, \xi^\prime, \xi_n) \, \dd \xi_n\]
satisfies that
\begin{equation}
\label{id:HQ-HQnu}
H_Q (\tau, \xi^\prime) = H_Q^{|\nu|} (\tau, \xi^\prime) \qquad \forall (\tau, \xi^\prime) \in \R \times \R^{n-1} : \tau \neq - |\xi^\prime|^2.
\end{equation}
With this at hand, we have that
\begin{equation}
\label{id:TnuHQ}
\langle e^{- \nu \cdot \x}  T_\nu (e^{\nu \cdot \x} G), \phi \rangle = \int_{\R \times \R^{n - 1}} H_Q (\tau, \xi^\prime) \, \dd (\tau, \xi^\prime). 
\end{equation}
As we will see at a later stage, the function
\begin{equation}
\label{map:GQ}
(\tau, \xi) \in \R \times \R^n \mapsto \frac{\widehat{G_Q} (\tau, \xi)}{- \tau - |\xi|^2} \, \widecheck{\phi_Q}(\tau, \xi)
\end{equation}
belongs to $L^1(\R \times \R^n)$. Hence, by the Fubini--Tonelli theorem we have 
that
\[ \int_{\R \times \R^{n - 1}} H_Q (\tau, \xi^\prime) \, \dd (\tau, \xi^\prime) = \int_{\R \times \R^n} \frac{\widehat{G_Q} (\tau, \xi)}{- \tau - |\xi|^2} \, \widecheck{\phi_Q}(\tau, \xi) \, \dd (\tau, \xi). \]
Using the identity \eqref{id:TnuHQ} and the change of variables that removes the 
orthogonal matrix $Q$, we have
\begin{equation}
\label{id:TnuG}
\langle e^{- \nu \cdot \x}  T_\nu (e^{\nu \cdot \x} G), \phi \rangle = \int_{\R \times \R^n} \frac{\widehat{G} (\tau, \xi)}{- \tau - |\xi|^2} \, \widecheck{\phi}(\tau, \xi) \, \dd (\tau, \xi)
\end{equation}
for all $\phi \in \mathcal{D}(\R \times \R^n)$ and 
$\nu \in \R^n \setminus \{ 0 \}$ such that $|\nu| < c$.
Let $v \in C(\R; L^2(\R^n))$ denote the function
\begin{equation*}
v (t, \centerdot) =
	\left\{
		\begin{aligned}
		& u (t, \centerdot) & & \textnormal{if} \, t \in [0,T], \\
		& 0 &  & \textnormal{if} \, t \in \R \setminus [0,T].
		\end{aligned}
	\right.
\end{equation*}
The continuity follows from the fact that 
$u (0, \centerdot) = u (T, \centerdot) = 0$. Using this,
one can check that
\[ (i\partial_\tm + \Delta) v = G \textnormal{ in } \R \times \R^n. \]
Taking Fourier transform we have that\footnote{The symbol $\dot{\forall}$ means 'for almost every'.}
\begin{equation}
\label{eq:RtimeRn}
(- \tau - |\xi|^2) \widehat{v}(\tau, \xi) = \widehat{G} (\tau, \xi) \enspace \dot{\forall} (\tau, \xi) \in \R \times \R^n.
\end{equation}
This identity and \eqref{id:TnuG} yield
\[\langle e^{- \nu \cdot \x}  T_\nu (e^{\nu \cdot \x} G), \phi \rangle = \langle v, \phi \rangle \]
for all $\phi \in \mathcal{D}(\R \times \R^n)$, or equivalently
\[\langle T_\nu (e^{\nu \cdot \x} G), \psi \rangle = \langle e^{\nu \cdot \x}  v, \psi \rangle \]
for all $\psi \in \mathcal{D}(\R \times \R^n)$ ---it is enough to consider 
$\phi = e^{\nu \cdot \x} \psi$ for an arbitrary $\psi$. In particular,
\[ T_\nu (e^{\nu \cdot \x} G)|_\Sigma =  e^{\nu \cdot \x}  u.\]
Then, the inequality \eqref{in:TnuExpG} is the one stated in the 
\cref{lem:exp_decay}.

In order to conclude this proof, it remains to check two points:
the function \eqref{map:GQ} belongs to $L^1(\R \times \R^n)$, and the identity
\eqref{id:HQ-HQnu} holds. Start by showing the first of these points.
It is clear that \eqref{map:GQ} belongs to $L^1(\R \times \R^n)$ if, and only if,
the function
\begin{equation}
\label{map:G}
(\tau, \xi) \in \R \times \R^n \mapsto \frac{\widehat{G} (\tau, \xi)}{- \tau - |\xi|^2} \, \widecheck{\phi}(\tau, \xi)
\end{equation}
also belongs to $L^1(\R \times \R^n)$. In order to check that \eqref{map:G}
belongs to $L^1(\R \times \R^n)$, 
we derive from the identity \eqref{eq:RtimeRn} that
\[ \widehat{v}(\tau, \xi) = \frac{\widehat{G} (\tau, \xi)}{- \tau - |\xi|^2} \]
for a.e. $ (\tau, \xi) \in \R \times \R^n$ such that $\tau \neq - |\xi|^2$.
Since the measure of the set 
$\{ (\tau, \xi) \in \R \times \R^n : \tau = - |\xi|^2 \}$ is zero, 
the above identity holds for a.e. $ (\tau, \xi) \in \R \times \R^n$.
This together with fact that $v \in L^2(\R \times \R^n)$ ensure that
\eqref{map:G} belongs to $L^1(\R \times \R^n)$.

Finally, let us check the identity \eqref{id:HQ-HQnu}. By the \cref{cor:vanish} 
and the \cref{lem:Fourier_hol-extension}
the function
\[ \zeta \in \C_{|\Im|<c} \mapsto \frac{\widehat{G_Q} (\tau, \xi^\prime, \zeta)}{- \tau - |\xi^\prime|^2 - \zeta^2} \, \widecheck{\phi_Q}(\tau, \xi^\prime, \zeta)\]
is holomorphic. Then, by the Cauchy--Goursat theorem we have that
\begin{align*}
H_Q (\tau, \xi^\prime) &+ \lim_{\rho \to \infty} i \int_0^{|\nu|} \frac{\widehat{G_Q} (\tau, \xi^\prime, \rho + it)}{- \tau - |\xi^\prime|^2 - (\rho + it)^2} \, \widecheck{\phi_Q}(\tau, \xi^\prime, \rho + it) \, \dd t \\
& = H_Q^{|\nu|} (\tau, \xi^\prime) + \lim_{\rho \to \infty} i \int_0^{|\nu|} \frac{\widehat{G_Q} (\tau, \xi^\prime, - \rho + it)}{- \tau - |\xi^\prime|^2 - (-\rho + it)^2} \, \widecheck{\phi_Q}(\tau, \xi^\prime, -\rho + it) \, \dd t 
\end{align*}
for all $(\tau, \xi^\prime) \in \R \times \R^{n-1}$ such that 
$ \tau \neq - |\xi^\prime|^2$. The item \eqref{ite:representation} of the 
\cref{lem:Fourier_hol-extension}, and the decay as $\rho$ tends to infinity,
can be used to show that each of the limits in
the previous identity is zero. Hence, the identity \eqref{id:HQ-HQnu} holds.
This concludes the proof of the \cref{lem:exp_decay}.
\end{proof}

A useful consequence of the \cref{lem:exp_decay} is the following integration
by parts that generalizes the \cref{lem:_generalization_integration_by_parts}
to admit exponentially growing solutions.

\begin{proposition}\label{prop:CGO-PhS_int-by-parts}\sl 
Consider $F \in L^2(\Sigma)$ such that 
there exists $c > 0$ so that $e^{c|\x|} F \in L^2(\Sigma)$.
Let $u \in C([0,T]; L^2(\R^n))$ satisfy the conditions
\[
\left\{
		\begin{aligned}
		& (i\partial_\tm + \Delta) u = F & & \textnormal{in} \enspace \Sigma, \\
		& u(0, \centerdot) = u(T, \centerdot) = 0 &  & \textnormal{in} \enspace \R^n.
		\end{aligned}
	\right.
\]
For every $\nu \in \R^n \setminus \{ 0 \}$ with $|\nu| < c$,
consider a measurable function 
$v^\nu : \R \times \R^n \rightarrow \C$ such that
\[ \sup_{\alpha \in \N_0^2} \big[ 2^{-|\alpha|/2} \| e^{- \nu \cdot \x} v^\nu \|_{L^2(\Pi_\alpha^\nu)} \big] + \sum_{\alpha \in \N_0^2} 2^{|\alpha| \theta} \| e^{- \nu \cdot \x} (i\partial_\tm + \Delta) v^\nu \|_{L^2(\Pi_\alpha^\nu)} < \infty, \]
for some $\theta \in (0, 1/2)$. Then,
\[\int_\Sigma (i\partial_\tm + \Delta) u \overline{v^\nu} = \int_\Sigma u \overline{(i\partial_\tm + \Delta)v^\nu} \, \]
for all $\nu \in \R^n \setminus \{ 0 \}$ with $|\nu| < c$.
\end{proposition}

\begin{proof}
Start by noticing that
\begin{equation}
\label{in:hiden_regularity}
\| \nu \cdot \nabla (e^{-\nu \cdot \x} v^\nu) \|_{L^2(\R \times \R^n)} \leq \sum_{\alpha \in \N_0^2} 2^{|\alpha| \theta} \| e^{- \nu \cdot \x} (i\partial_\tm + \Delta) v^\nu \|_{L^2(\Pi_\alpha^\nu)}.
\end{equation}
Indeed, a simple computation shows that
\[e^{- \nu \cdot \x} (i\partial_\tm + \Delta) v^\nu = (i\partial_\tm + \Delta + 2 \nu \cdot \nabla + |\nu|^2) (e^{- \nu \cdot \x} v^\nu).\]
The symbol of the differential operator on the right-hand side of 
the previous identity is 
$p_{(0, \nu)} = - \tau - |\xi|^2 + i 2\nu \cdot \xi + |\nu|^2$.
By Plancherel's identity we have that
\[ \| e^{- \nu \cdot \x} (i\partial_\tm + \Delta) v^\nu \|_{L^2(\R \times \R^n)} \geq 
\| \Im p_{(0, \nu)} \widehat{e^{- \nu \cdot \x} v^\nu} \|_{L^2(\R \times \R^n)} \geq \| \nu \cdot \nabla (e^{-\nu \cdot \x} v^\nu) \|_{L^2(\R \times \R^n)}, \]
where $\Im p_{(0, \nu)} $ denotes the imaginary part of $p_{(0, \nu)}$.
In order to conclude \eqref{in:hiden_regularity}, it is enough to observe that
\[ \| e^{- \nu \cdot \x} (i\partial_\tm + \Delta) v^\nu \|_{L^2(\R \times \R^n)} \leq  \sum_{\alpha \in \N_0^2} 2^{|\alpha| \theta} \| e^{- \nu \cdot \x} (i\partial_\tm + \Delta) v^\nu \|_{L^2(\Pi_\alpha^\nu)}. \]

From the inequality \eqref{in:exp_decay}, we have
\[ \sum_{j \in \N_0 } 2^j \| e^{\nu \cdot \x} F \|_{L^2(\Sigma_j^\nu)} \leq \Big( 1 +  \sum_{j \in \N } 2^j e^{-( \frac{c}{|\nu|} - 1 )2^{j - 1}} \Big) \| e^{c|\x|} F \|_{L^2(\Sigma)}, \]
where the series between parenthesis is convergent.
This together with the fact that $F = (i\partial_\tm + \Delta) u$ and the assumption on $v^\nu$ make it possible the 
use of the dominate convergence theorem to conclude that
\begin{equation}
\label{id:DCT}
\int_\Sigma (i\partial_\tm + \Delta) u \overline{v^\nu}\, = \lim_{R \to \infty} \int_\Sigma (i\partial_\tm + \Delta) u \, \overline{\chi_R^\nu v^\nu}
\end{equation}
where $\chi^\nu_R (x) = \chi (\nu \cdot x/R)$ for all $x \in \R^n$ with 
$\chi \in \mathcal{S}(\R)$ such that $0 \leq \chi(s) \leq 1$ for all $s \in \R$, 
$\supp \chi \subset \{ s \in \R : |s| \leq 1 \}$ and $\chi(s) = 1$ 
whenever $|s| \leq 1/2$.

We want to apply the integration-by-parts formula proved in the
\cref{lem:_generalization_integration_by_parts}, for that we need to show
that the restrictions of $\chi_R^\nu v^\nu $ and 
$(i\partial_\tm + \Delta) (\chi_R^\nu v^\nu) $ to 
$\tilde{\Sigma} = (-\delta, T + \delta) \times \R^n$ belongs to
$L^2(\tilde{\Sigma})$. The requirement for $\chi_R^\nu v^\nu $ is clearly true,
so we just need to check the requirement for
$(i\partial_\tm + \Delta) (\chi_R^\nu v^\nu) $. Leibniz's rule implies that
\begin{equation}
\label{id:Leibniz_rule}
(i\partial_\tm + \Delta) (\chi_R^\nu v^\nu) = \chi_R^\nu (i\partial_\tm + \Delta) v^\nu + \frac{2}{R} \chi^\prime (\nu \cdot \x /R) \nu \cdot \nabla v^\nu + \frac{|\nu|^2}{R^2} \chi^{\prime \prime} (\nu \cdot \x /R) v^\nu,
\end{equation}
where
\begin{equation}
\label{id:Leibniz_exp}
\nu \cdot \nabla v^\nu = e^{\nu \cdot \x} \nu \cdot \nabla (e^{- \nu \cdot \x} v^\nu) + |\nu|^2 v^\nu.
\end{equation}
By the hidden regularity exhibited in \eqref{in:hiden_regularity} and the 
previous computations is now clear that the restriction to $\tilde{\Sigma}$
of $(i\partial_\tm + \Delta) (\chi_R^\nu v^\nu) $ belongs to
$L^2(\tilde{\Sigma})$. Therefore, we can apply the 
\cref{lem:_generalization_integration_by_parts} to the integral inside the limit
of the identity \eqref{id:DCT}:
\[\int_\Sigma (i\partial_\tm + \Delta) u \overline{v^\nu}\, = \lim_{R \to \infty} \int_\Sigma u \, \overline{(i\partial_\tm + \Delta) (\chi_R^\nu v^\nu)}.\]
The identities \eqref{id:Leibniz_rule} and \eqref{id:Leibniz_exp} give rise to
\begin{align*}
\int_\Sigma u \, \overline{(i\partial_\tm + \Delta) (\chi_R^\nu v^\nu)} =& \int_\Sigma \chi_R^\nu\, u \, \overline{(i\partial_\tm + \Delta) v^\nu} + \frac{2}{R} \int_\Sigma \chi^\prime (\nu \cdot \x /R) e^{\nu \cdot \x} u \, \overline{ \nu \cdot \nabla (e^{- \nu \cdot \x} v^\nu)}\\
& + |\nu|^2 \int_\Sigma \Big[ \frac{2}{R} \chi^\prime (\nu \cdot \x /R) + \frac{1}{R^2} \chi^{\prime \prime} (\nu \cdot \x /R) \Big] u \, \overline{v^\nu}.
\end{align*}
From the assumption for $v^\nu$ and the \cref{lem:exp_decay}, we have by the 
dominate convergence theorem that
\[ \lim_{R \to \infty} \int_\Sigma \chi_R^\nu\, u \, \overline{(i\partial_\tm + \Delta) v^\nu} = \int_\Sigma u \, \overline{(i\partial_\tm + \Delta) v^\nu}. \]
Additionally, using the \cref{lem:exp_decay} to estimate $u$ as $R$ becomes 
large, the hidden-regularity inequality \eqref{in:hiden_regularity}, and the 
assumptions on $v^\nu$, one can check that
\[ \lim_{R \to \infty} \frac{1}{R} \int_\Sigma \chi^\prime (\nu \cdot \x /R) e^{\nu \cdot \x} u \, \overline{ \nu \cdot \nabla (e^{- \nu \cdot \x} v^\nu)} \, = 0, \]
and
\[ \lim_{R \to \infty} \frac{1}{R} \int_\Sigma \chi^\prime (\nu \cdot \x /R) u \, \overline{v^\nu} \, = \lim_{R \to \infty} \frac{1}{R^2} \int_\Sigma \chi^{\prime \prime} (\nu \cdot \x /R) u \, \overline{v^\nu} \, = 0. \]
Thus, the integration-by-parts formula of the statement is proved.
\end{proof}

\subsection{An orthogonal relation for exponentially growing solutions}
Whenever $V_1$ and $V_2$ are so that their corresponding initial-to-final-state
maps $\mathcal{U}_T^1$ and $\mathcal{U}_T^2$ satisfy
\[\mathcal{U}_T^1 = \mathcal{U}_T^2\]
The integral identity \eqref{id:integral_identity} becomes an orthogonality
relation between $V_1 - V_2$ and the product $u_1 \overline{v_2}$, which are 
physical solutions. The goal here is to prove that the same orthogonal relation
holds for the product of exponentially growing solutions.

\begin{theorem}\label{th:orthogolaity_CGO} \sl 
Consider $V_1, V_2 \in L^1((0, T); L^\infty(\R^n))$ such that there exists $c > 0$
so that $e^{c |\x|} V_j \in L^\infty (\Sigma) \cap L^2 (\Sigma)$
for $j \in \{ 1, 2 \}$.
Let $\mathcal{U}_T^1$ and $\mathcal{U}_T^2$ denote the initial-to-final-state
maps associated to $V_1$ and $V_2$ respectively. Then, the equality 
$\mathcal{U}_T^1 = \mathcal{U}_T^2$ implies that
\[ \int_\Sigma (V_1 - V_2)u_1^\eta \overline{v_2^\nu}\, = 0 \]
for all $u_1^\eta$ and $v_2^\nu$
solutions of
\[ (i\partial_\tm + \Delta - V_1^{\rm ext}) u_1^\eta =  (i\partial_\tm + \Delta - \overline{V_2^{\rm ext}}) v_2^\nu = 0 \enspace \textnormal{in} \enspace \R \times \R^n, \]
with $V_j^{\rm ext}$ denoting the trivial extension \eqref{def:extension} 
of $V_j$, and such that
\[ \sup_{\alpha \in \N_0^2} \big[ 2^{-|\alpha|/2} \| e^{-\eta \cdot \x} u_1^\eta \|_{L^2(\Pi_\alpha^\eta)} \big] + \sup_{\alpha \in \N_0^2} \big[ 2^{-|\alpha|/2} \| e^{- \nu \cdot \x} v_2^\nu \|_{L^2(\Pi_\alpha^\nu)} \big] < \infty \]
for $ \eta, \nu \in \R^n \setminus \{ 0 \}$ with $|\eta| + |\nu| < c$.
\end{theorem}

\begin{proof}
We start by proving that if $\mathcal{U}_T^1 = \mathcal{U}_T^2$ then
\[ \int_\Sigma (V_1 - V_2)u_1 \overline{v_2^\nu}\, = 0 \]
for every $u_1$ physical solution of \eqref{pb:IVP} with potential $V_1$,
and every exponentially growing solution $v_2^\nu $,
with $\nu \in \R^n \setminus \{ 0 \}$ and $|\nu| < c$, as in the statement.

For every physical solution $u_1$ of \eqref{pb:IVP} with potential $V_1$,
let $w_2 \in C([0, T]; L^2(\R^n))$ be the solution
of the problem 
\[\left\{
		\begin{aligned}
		& (i\partial_\tm + \Delta -  V_2)w_2 = (V_1 - V_2)u_1 & & \textnormal{in} \, \Sigma, \\
		& w_2(0, \centerdot) = 0 &  & \textnormal{in} \, \R^n.
		\end{aligned}
\right.\]
Since $\mathcal{U}_T^1 = \mathcal{U}_T^2$
and $ V_1, V_2 \in L^1((0,T); L^\infty (\R^n))$,
we can apply the \cref{lem:orthogonal_rewardPB} with $F = (V_1 - V_2)u_1$
and \eqref{id:integral_identity} to deduce that $w_2(T, \centerdot) = 0$.
Furthermore, since
$e^{c |\x|} V_j \in L^\infty (\Sigma)$ for $j \in \{ 1, 2 \}$,
we have that $e^{c |\x|} (i\partial_\tm + \Delta)w_2 \in L^2 (\Sigma) $.
This means that $w_2$ shares
the same properties as $u$ in the \cref{prop:CGO-PhS_int-by-parts} for all
$\nu \in \R^n \setminus \{ 0\}$ with $|\nu| < c$.

Furthermore, every exponentially growing solution $v_2^\nu$ satisfies the same conditions as $v^\nu$ in the \cref{prop:CGO-PhS_int-by-parts}. Indeed, the fact
that
\[ \sum_{j \in \N_0} 2^{j(1/2 + \theta)} \| V_2 \|_{L^\infty(\Sigma^\nu_j)} \leq \Big( 1 + \sum_{j \in \N} 2^{j(\theta + 1/2)} e^{-c2^{j-1}} \Big) \| e^{c |x|} V_2 \|_{L^\infty (\Sigma)} \]
for all $\theta \in (0, 1/2)$ and $\nu \in \R^n \setminus \{ 0\}$, implies that
\[ \sum_{\alpha \in \N_0^2} 2^{|\alpha| \theta} \| e^{- \nu \cdot \x} (i\partial_\tm + \Delta) v_2^\nu \|_{L^2(\Pi_\alpha^\nu)} < \infty \]
for all $\theta \in (0, 1/2)$ and $\nu \in \R^n \setminus \{ 0\}$ with 
$|\nu| < c$.

Then, $w_2$ and $v_2^\nu$ satisfy the hypothesis of the 
\cref{prop:CGO-PhS_int-by-parts}, for all $\nu \in \R^n \setminus \{ 0\}$ with 
$|\nu| < c$, and we have the 
following integration-by-parts formula
\[\int_\Sigma (i\partial_\tm + \Delta) w_2 \overline{v_2^\nu} = \int_\Sigma w_2 \overline{(i\partial_\tm + \Delta)v_2^\nu} \,. \]
The inequality 
\[ \sum_{j \in \N_0} 2^{j/2} \| e^{\nu \cdot \x} (V_1 - V_2) \|_{L^\infty(\Sigma^\nu_j)} \leq \Big( 1 + \sum_{j \in \N} 2^{j/2} e^{-( \frac{c}{|\nu|} - 1 )2^{j - 1}} \Big) \| e^{c |x|} (V_1 - V_2) \|_{L^\infty (\Sigma)} \]
guarantees that 
$(V_1 - V_2)u_1 \overline{v_2^\nu} \in L^1(\Sigma)$ for all 
$\nu \in \R^n \setminus \{ 0\}$ with $|\nu| < c$. 
That inequality follows from \eqref{in:exponents} and the series between 
parenthesis converges if $|\nu| < c$.
As a consequence of the fact that
$(V_1 - V_2)u_1 \overline{v_2^\nu} \in L^1(\Sigma)$ we can write
\[ \int_\Sigma (V_1 - V_2)u_1 \overline{v_2^\nu}\, = \int_\Sigma (i\partial_\tm + \Delta - V_2)w_2 \overline{v_2^\nu}\, = \int_\Sigma w_2 \overline{(i\partial_\tm + \Delta - \overline{V_2}) v_2^\nu}\, = 0. \]
In the first identity we have used that
$(i\partial_\tm + \Delta - V_2)w_2 = (V_1 - V_2)u_1$, in the second one we have
used the previous integration-by-parts formula for $w_2$ and $v_2^\nu$,
and in the last equality we have used that $(i\partial_\tm + \Delta - \overline{V_2}) v_2^\nu = 0$ in $\Sigma$. This proves that
if $\mathcal{U}_T^1 = \mathcal{U}_T^2$ then
\begin{equation}
\label{id:orthogonal_phy-exp}
\int_\Sigma (V_1 - V_2)u_1 \overline{v_2^\nu}\, = 0
\end{equation}
for every $u_1$ physical solution with potential $V_1$,
and every exponentially growing solution $v_2^\nu$,
with $\nu \in \R^n \setminus \{ 0 \}$ and $|\nu| < c$, as in the statement.

Finally, we will prove that the orthogonality relation 
\eqref{id:orthogonal_phy-exp} yields 
the one in the statement, that is,
\[ \int_\Sigma (V_1 - V_2)u_1^\eta \overline{v_2^\nu}\, = 0 \]
for every exponentially growing solutions $u_1^\eta$ and $v_2^\nu$ as in the 
statement.

Start by noticing that the inequality
\[ \sum_{j \in \N_0} 2^{j/2} \| e^{\nu \cdot \x} (V_1 - V_2) \|_{L^2(\Sigma^\nu_j)} \leq \Big( 1 + \sum_{j \in \N} 2^{j/2} e^{-( \frac{c}{|\nu|} - 1 )2^{j - 1}} \Big) \| e^{c |x|} (V_1 - V_2) \|_{L^2 (\Sigma)} \]
ensures that
$(\overline{V_1} - \overline{V_2}) v_2^\nu $ belongs to $ L^1((0,T); L^2(\R^n))$
for all $\nu \in \R^n \setminus \{ 0\}$ with $|\nu| < c$.
Again, that inequality follows from \eqref{in:exponents} and the series between 
parenthesis converges if $|\nu| < c$.
Then, let $w_1^\nu \in C([0, T]; L^2(\R^n))$ be the solution
of the problem
\[\left\{
		\begin{aligned}
		& (i\partial_\tm + \Delta - \overline{V_1})w_1^\nu = (\overline{V_1} - \overline{V_2}) v_2^\nu & & \textnormal{in} \, \Sigma, \\
		& w_1(T, \centerdot) = 0 &  & \textnormal{in} \, \R^n.
		\end{aligned}
\right.\]
Since $ V_1 \in L^1((0,T); L^\infty (\R^n))$ and $(\overline{V_1} - \overline{V_2}) v_2^\nu \in L^1((0,T); L^2(\R^n))$,
we can apply the \cref{lem:orthogonal_forwardPB} with 
$G = (\overline{V_1} - \overline{V_2}) v_2^\nu$ and \eqref{id:orthogonal_phy-exp}
to deduce that $w_1^\nu(0, \centerdot) = 0$.
As we argued earlier to prove \eqref{id:orthogonal_phy-exp}, we want to show that,
given $\nu \in \R^n \setminus \{ 0 \}$ with  $|\nu| < c$, we have that
$ e^{c^\prime |\x|} (i\partial_\tm + \Delta)w_1^\nu \in L^2(\Sigma) $
for all $c^\prime \in (0, c - |\nu|)$.
In this way, we could conclude that $w_1^\nu$,
with $\nu \in \R^n \setminus \{ 0\}$ and $|\nu| < c$, shares
the same properties as $u$ in the \cref{prop:CGO-PhS_int-by-parts} for all 
$c^\prime < c - |\nu|$.

Since $e^{c |\x|} \overline{V_1} w_1^\nu \in L^2(\Sigma)$, in order to prove that
$ e^{c^\prime |\x|} (i\partial_\tm + \Delta)w_1^\nu \in L^2(\Sigma) $
for all $c^\prime \in (0, c - |\nu|)$, we only have to check that
$ e^{c^\prime |\x|} (\overline{V_1} - \overline{V_2}) v_2^\nu \in L^2(\Sigma) $
for all $c^\prime < c - |\nu|$. This follows from the inequality
\begin{align*}
\sum_{j \in \N_0} 2^{j/2} \| e^{c^\prime |\x| + \nu \cdot \x} (V_1 & - V_2) \|_{L^\infty(\Sigma^\nu_j)} \\
& \leq \Big( 1 + \sum_{j \in \N} 2^{j/2} e^{-( \frac{c - c^\prime}{|\nu|} - 1 )2^{j - 1}} \Big) \| e^{c |x|} (V_1 - V_2) \|_{L^\infty (\Sigma)},
\end{align*}
where the series between parenthesis converges if $c^\prime < c - |\nu|$.
To justify the above inequality we have used that
\[\nu \cdot x \leq - \Big( \frac{c - c^\prime}{|\nu|} - 1 \Big) |\nu \cdot x| +  \frac{c - c^\prime}{|\nu|} |\nu \cdot x| \leq - \Big( \frac{c - c^\prime}{|\nu|} - 1 \Big) |\nu \cdot x| +  (c - c^\prime) |x|,\]
which is valid for all $x \in \R^n$. This discussion was to prove that
$ e^{c^\prime |\x|} (i\partial_\tm + \Delta)w_1^\nu \in L^2(\Sigma) $
for all $c^\prime \in (0, c - |\nu|)$ with $\nu \in \R^n \setminus \{ 0\}$ and
$|\nu| < c$.

Additionally, every exponentially growing solution $u_1^\eta$ satisfies the same 
conditions as $v^\eta$ in the \cref{prop:CGO-PhS_int-by-parts}.
Indeed, the fact that 
\[ \sum_{j \in \N_0} 2^{j(1/2 + \theta)} \| V_1 \|_{L^\infty(\Sigma^\eta_j)} \leq \Big( 1 + \sum_{j \in \N} 2^{j(\theta + 1/2)} e^{-c2^{j-1}} \Big) \| e^{c |x|} V_1 \|_{L^\infty (\Sigma)} \]
for all $\theta \in (0, 1/2)$ and $\eta \in \R^n \setminus \{ 0\}$, 
implies that
\[ \sum_{\alpha \in \N_0^2} 2^{|\alpha| \theta} \| e^{-\eta \cdot \x} (i\partial_\tm + \Delta) u_1^\eta \|_{L^2(\Pi_\alpha^\eta)} < \infty \]
for all $\theta \in (0, 1/2)$.

Then, given $\nu \in \R^n \setminus \{ 0 \}$ with $|\nu| < c$, we have $w_1^\nu$ 
and $u_1^\eta$, with $\eta \in \R^n \setminus \{ 0 \}$ and $|\eta| < c^\prime$, 
satisfy the hypothesis of the \cref{prop:CGO-PhS_int-by-parts} for all
$c^\prime \in (0, c - |\nu|)$. Thus, we have that the integration-by-parts formula
\[\int_\Sigma (i\partial_\tm + \Delta) w_1^\nu \overline{u_1^\eta} = \int_\Sigma w_1^\nu \overline{(i\partial_\tm + \Delta)u_1^\eta} \, \]
holds for all $ \eta, \nu \in \R^n \setminus \{ 0 \}$ with $|\eta| + |\nu| < c$.

The inequality
\begin{align*}
& \sum_{\alpha \in \N_0^2} 2^{|\alpha|/2} \| e^{(\eta + \nu) \cdot \x} (V_1 - V_2) \|_{L^\infty(\Sigma^\nu_{\alpha_1} \cap \Sigma^\eta_{\alpha_2})} \\
& \enspace \leq \Big( 1 + \sum_{j \in \N} 2^\frac{j}{2} e^{-( \frac{c - |\eta|}{|\nu|} - 1 )2^{j - 1}} \Big) \Big( 1 + \sum_{k \in \N} 2^\frac{k}{2} e^{-( \frac{c - |\nu|}{|\eta|} - 1 )2^{k - 1}} \Big) \| e^{c |x|} (V_1 - V_2) \|_{L^\infty (\Sigma)} 
\end{align*}
guarantees that $(V_1 - V_2)u_1^\eta \overline{v_2^\nu} \in L^1(\Sigma)$, since
the series between parenthesis converge if $|\eta| + |\nu| < c$.
Recall that $\alpha = (\alpha_1, \alpha_2)$ and $|\alpha| = \alpha_1 + \alpha_2$.
To prove the above inequality we have used that
\[(\eta + \nu) \cdot x \leq - \Big( \frac{c - |\eta|}{|\nu|} - 1 \Big) |\nu \cdot x| - \Big( \frac{c - |\nu|}{|\eta|} - 1 \Big) |\eta \cdot x| +  (c - |\eta| - |\nu|) |x|,\]
which is valid for all $x \in \R^n$, whenever $|\eta| < c$ and $|\nu| < c$.
Since $(V_1 - V_2)u_1^\eta \overline{v_2^\nu} \in L^1(\Sigma)$ whenever
$|\eta| + |\nu| < c$, we can write
\[ \int_\Sigma (V_1 - V_2)u_1^\eta \overline{v_2^\nu}\, = \int_\Sigma u_1^\eta \overline{(i\partial_\tm + \Delta - \overline{V_1}) w_1^\nu} \, = \int_\Sigma (i\partial_\tm + \Delta - V_1) u_1^\eta \overline{w_1^\nu} \, = 0. \]
In the first identity we have used that
$(i\partial_\tm + \Delta - \overline{V_1})w_1^\nu = (\overline{V_1} - \overline{V_2}) v_2^\nu$, in the second one we have
used the previous integration-by-parts formula for $w_1^\nu$ and $u_1^\eta$,
and in the last equality we have used that $(i\partial_\tm + \Delta - V_1) u_1^\eta = 0$ in $\Sigma$. This proves that
the orthogonality relation
\[ \int_\Sigma (V_1 - V_2)u_1^\eta \overline{v_2^\nu}\, = 0 \]
holds for every every exponentially growing solutions $u_1^\eta$ and $v_2^\nu$,
with $\eta, \nu \in \R^n \setminus \{ 0 \}$ and $|\eta| + |\nu| < c$,
as in the statement.
\end{proof}

\section{Proof of uniqueness} \label{sec:uniqueness}
Consider $V_1, V_2 \in L^1((0, T); L^\infty (\R^n))$ with super-exponential decay.
Let $V_1^{\rm ext}$ and $V_2^{\rm ext}$ denote the trivial extensions
\[V_k^{\rm ext} (t, x) = \left\{ \begin{aligned}
& V_k (t, x) & & (t, x) \in \Sigma,\\
& 0 & & (t, x) \notin \Sigma.
\end{aligned} \right.\]
For $\nu \in \R^n \setminus \{ 0 \}$, choose
\[\varphi_1(t,x) = i |\nu|^2 t - \nu \cdot x , \quad \varphi_2(t,x) = i |\nu|^2 t + \nu \cdot x \qquad \forall (t, x) \in \R \times \R^n.\]
Let $u_1 = e^{\varphi_1} (u_1^\sharp + u_1^\flat)$ and 
$v_2 = e^{\varphi_2} (v_2^\sharp + v_2^\flat)$ denote the CGO solutions of
the \cref{th:CGO} for the equations
\[ (i\partial_\tm + \Delta - V_1^{\rm ext}) u_1 = 0 \enspace \textnormal{and} \enspace (i\partial_\tm + \Delta - \overline{V_2^{\rm ext}}) v_2 = 0 \enspace \textnormal{in} \enspace \R \times \R^n, \]
respectively. For convenience, recall that for every 
$(t, x) \in \R \times \R^n$,
\[ u_1^\sharp (t, x) = [e^{i t \Delta^\prime} f] (QRe_1 \cdot x, \dots, QRe_{n - 1} \cdot x), \, v_2^\sharp (t, x) = [e^{i t \Delta^\prime} g] (Qe_1 \cdot x, \dots, Qe_{n - 1} \cdot x),\]
where $\Delta^\prime$ is the free Hamiltonian in
$\R^{n - 1}$, $f$ and $ g$ are arbitrary functions in $L^2(\R^{n - 1})$, and 
$Q , R \in \Orth(n) $ such that $\nu = |\nu| Qe_n$, $ R e_j = e_j $ for
$j \in \{ 1, \dots , n - 1 \}$ and $R e_n = - e_n$ (note that 
$-\nu = |\nu| Q R e_n$).
Furthermore, whenever $\nu \in \R^n \setminus \{ 0 \}$ with $|\nu| \geq \rho_0$ we have that
\[ \sup_{\alpha \in \N_0^2} \big[ 2^{-|\alpha|/2} \| u_1^\sharp \|_{L^2 (\Pi^\nu_\alpha)} \big] + |\nu|^{1/2} \sup_{\alpha \in \N_0^2} \big[ 2^{-|\alpha|(1/2 - \theta)} \| u_1^\flat \|_{L^2 (\Pi^\nu_\alpha)} \big] \lesssim \| f \|_{L^2(\R^{n - 1})} \]
and
\[ \sup_{\alpha \in \N_0^2} \big[ 2^{-|\alpha|/2} \| v_2^\sharp \|_{L^2 (\Pi^\nu_\alpha)} \big] + |\nu|^{1/2} \sup_{\alpha \in \N_0^2} \big[ 2^{-|\alpha|(1/2 - \theta)} \| v_2^\flat \|_{L^2 (\Pi^\nu_\alpha)} \big] \lesssim \| g \|_{L^2(\R^{n - 1})}. \]
By the \cref{th:orthogolaity_CGO}, we have that
\[ \Big| \int_\Sigma (V_1 - V_2) u_1^\sharp \overline{v_2^\sharp} \, \Big| \lesssim  \frac{1}{|\nu|^{1/2}} \| f \|_{L^2(\R^{n - 1})} \| g \|_{L^2(\R^{n - 1})}, \]
where the implicit constant only depends on $V_1$, $V_2$, $T$ and $\rho_0$.
Making $|\nu|$ go to infinity. We deduce that
\[ \int_\Sigma (V_1 - V_2) u_1^\sharp \overline{v_2^\sharp} \, = 0. \]
Write $F = V_1^{\rm ext} - V_2^{\rm ext}$, since $ R e_j = e_j $ for
$j \in \{ 1, \dots , n - 1 \}$ we have
\[ 0 = \int_{\R \times \R^n} F u_1^\sharp \overline{v_2^\sharp} \, =
\int_{\R \times \R^n} F (t, Qy) \,e^{i t \Delta^\prime} f (y^\prime) \, \overline{e^{i t \Delta^\prime} g (y^\prime)} \, \dd (t, y) \]
for all $f, g \in L^2(\R^{n - 1})$ and all $Q \in \Orth(n)$ 
---here $y = (y^\prime, y_n) \in \R^{n - 1} \times \R$. Then,
by a density argument similar to the one after \eqref{id:orthogonality_2nd},
we have that
\[ \int_{\R \times \R^n} F (t, Qy) \, e^{-i|\eta^\prime |^2 t + i \eta^\prime \cdot y^\prime} \, \overline{e^{-i|\kappa^\prime|^2 t + i \kappa^\prime \cdot y^\prime}} \, \dd (t, y) = 0 \]
for all $ \eta^\prime, \kappa^\prime \in \R^{n - 1}$ and
all $Q \in \Orth(n)$. Hence,
\[ \int_{\R \times \R^n} F (t, x) \, e^{-i|Q \eta|^2 t + i Q \eta \cdot x} \, \overline{e^{-i|Q \kappa|^2 t + i Q \kappa \cdot x}} \, \dd (t, x) = 0 \]
for all $\eta, \kappa \in \R^{n}$ such that 
$\eta \cdot e_n = \kappa \cdot e_n =0$, 
and all $Q \in \Orth(n)$.

Given $(\tau, \xi) \in \R \times \R^n$ such that $\xi \neq 0$,
we choose $\nu \in \R^n \setminus \{ 0 \}$ so that 
$\xi \cdot \nu = 0$ and, $\eta, \kappa \in \R^{n}$ such that 
$\eta \cdot e_n = \kappa \cdot e_n =0$ satisfying
\[Q \eta = -\frac{1}{2} \Big( 1 + \frac{\tau}{|\xi|^2} \Big) \xi, \qquad Q \kappa = \frac{1}{2} \Big( 1 - \frac{\tau}{|\xi|^2} \Big) \xi,\]
where $Q \in \Orth(n)$ so that $\nu = |\nu| Q e_n$.
Since
\[ -i(|Q \eta|^2 - |Q \kappa|^2) t - i (Q \kappa - Q \eta) \cdot x = -i \tau t - i \xi \cdot x, \]
we have that
\[ \widehat{F} (\tau, \xi) = 0 \]
for all $(\tau, \xi) \in \R \times (\R^n \setminus \{ 0 \}) $.
Additionally,
if $(\tau, \xi) = (0, 0)$ we choose an arbitrary $\nu \in \R^n \setminus \{ 0 \}$ and $\eta = \kappa = 0$, and we obtain that
$\widehat{F}(0, 0) = 0$. Consequently, $\widehat{F}$ vanishes in 
$\{ (\tau, \xi) \in \R \times \R^n : \xi \neq 0 \} \cup \{ (0, 0) \}$.
Since the $\widehat{F}$ is 
continuous in $\R \times \R^n$ we can conclude that $\widehat{F}(\tau, \xi) = 0$ 
for all $(\tau, \xi) \in \R \times \R^n$. Consequently, by the injectivity of the 
Fourier transform we have that $F(t, x) = 0$ for a.e. 
$(t, x) \in \R \times \R^n$, which implies that $V_1 (t, x) = V_2(t, x)$ for
a.e. $(t, x) \in \Sigma$. This concludes the proof of the 
\cref{th:uniqueness}.

\begin{acknowledgements}
P.C. is supported by the project PID2021-122156NB-I00, as well as BCAM-BERC 
2022-2025 and the BCAM Severo Ochoa CEX2021-001142-S.
\end{acknowledgements}

\bibliography{references}{}
\bibliographystyle{plain}

\end{document}